\documentclass[a4paper,reqno]{amsart}
\usepackage{mathtools,amsmath,amssymb,amsfonts,amsthm}
\usepackage{xcolor}
\usepackage[a4paper]{geometry}
\usepackage{hyperref}
\usepackage[T1]{fontenc} 
\usepackage[utf8]{inputenc}
\usepackage{cancel}
\usepackage{bm, bbm}
\usepackage{tikz, pgf}
\usepackage{fancyhdr}
\newcommand\R{\mathbb{R}}
\newcommand\Rn{\mathbb{R}^d}

\newcommand\N{\mathbb{N}}
\newcommand{\BV}{\mathrm{BV}}
\newcommand{\SBV}{\mathrm{SBV}}
\newcommand{\SBVG}{\mathrm{SBV}_{\mathcal{G}}}

\newcommand{\ebeta}{{[\beta]}}

\newcommand{\dX}{\mathrm{d}_X}
\newcommand{\dG}{\mathrm{d}_{\G}}

\newcommand\calH{\mathcal{H}}

\newcommand\calL{\mathcal{L}}

\newcommand\calU{\mathcal{U}}

\newcommand\tv{\widetilde{v}}

\newcommand\toG{\xrightarrow{{L}^1_\G}}
\newcommand\LG{{L}^1_\G}

\newcommand\dist{{\mathrm{dist}}}

\newcommand\Id{{\mathrm{Id}}}
\newcommand\SO{{\mathrm{SO}}}
\newcommand\sym{{\mathrm{sym}}}

\newcommand\loc{{\mathrm{loc}}}
\newcommand\out{{\mathrm{out}}}
\newcommand\iin{{\mathrm{in}}}
\newcommand\eps{{\varepsilon}}
\newcommand\di{{\,\mathrm{d}}}
\newcommand\dix{{\,\mathrm{d}_X}}

\newcommand\dt{{\,\mathrm{d}t}}

\newcommand{\Per}{\mathrm{Per}}
\newcommand{\LM}[1]{\hbox{\vrule width.2pt \vbox to#1pt{\vfill \hrule width#1pt height.2pt}}}
\newcommand{\LL}{{\mathchoice
{\,\LM7\,}{\,\LM7\,}{\,\LM5\,}{\,\LM{3.35}\,}}}

\newtheorem{definition}{Definition}[section]
\newtheorem{lemma}[definition]{Lemma}
\newtheorem{theorem}[definition]{Theorem}
\newtheorem{proposition}[definition]{Proposition}

\theoremstyle{remark}
\newtheorem{remark}[definition]{Remark}

\def\e{\varepsilon}

\usepackage{color}
\usepackage{hyperref}
\usepackage[T1]{fontenc} 
\usepackage[utf8]{inputenc}
\definecolor{ddorange}{rgb}{1,0.5,0}
\definecolor{ddcyan}{rgb}{0,0.2,1.0}

\newcommand{\EEE}{\color{black}}

\newcommand\dx{{\,\mathrm{d}x}}
\newcommand\dy{{\,\mathrm{d}y}}
\newcommand\ds{{\,\mathrm{d}s}}
\newcommand{\ol}{\overline}
\newcommand{\sm}{\setminus}

\newcommand{\F}{\mathcal{F}}
\newcommand{\G}{\mathcal{G}}
\newcommand{\cH}{\mathcal{H}}
\newcommand{\FT}{\mathrm{FT}}

\newcommand{\Fr}{\F^{\mathrm{red}}}

\newcommand{\tF}{\widetilde{\F}}

\newcommand{\SOn}{\mathrm{SO}(d)}
\newcommand{\On}{\mathrm{O}(d)}

\newcommand{\OnG}{{\On}/{\G}}

\newcommand{\Mno}{\mathrm{M}^{d\times d}}
\newcommand{\Mdue}{\mathrm{M}^{2\times 2}}
\newcommand{\Mns}{\mathrm{M}^{d\times d}_{\mathrm{sym}}}
\newcommand{\Mnsp}{\mathrm{M}^{d\times d}_{\mathrm{sym},+}}

\newcommand{\PC}{\mathrm{SBV}_0}

\newcommand{\xy}{^\xi_y}

\newcommand{\LN}{\calL^N}
\newcommand{\hN}{\calH^{N-1}}
\newcommand{\SN}{\mathbb{S}^{N-1}}
\newcommand{\Lip}{\mathrm{Lip}}
\newcommand{\dhN}{{\,\mathrm{d}\calH^{N-1}}}

\newcommand\uu{{u}}

\newcommand\AT{\text{PF}}

\numberwithin{equation}{section}

\title[Approximation of sharp-interface energies accounting for lattice symmetry]{Phase-field approximation of sharp-interface energies accounting for lattice symmetry}
\author{Sergio Conti, Vito Crismale, Adriana Garroni, Annalisa Malusa}
\address{Sergio Conti: Institut f\"ur Angewandte Mathematik, Universit\"at Bonn, 53115 Bonn, Germany \\
Vito Crismale, Adriana Garroni, Annalisa Malusa: Dipartimento di Matematica, Università La Sapienza di Roma, 
piazzale Aldo Moro 5, 00185 Roma}
\date{\today}

\begin{document}
\maketitle
\begin{abstract}
	We present a phase-field approximation of sharp-interface  energies defined on partitions, designed for modeling grain boundaries in polycrystals. The independent variable takes values in
the orthogonal group $\On$ modulo a lattice point group $\G$, reflecting the crystallographic symmetries of the underlying lattice. In the sharp-interface limit, the surface energy
exhibits a Read-Shockley-type behavior for small misorientation angles, scaling as $\theta|\log\theta|$.
The regularized functionals are applicable to grain growth simulation and the reconstruction of grain boundaries from imaging data.
\end{abstract}
\tableofcontents

\section{Introduction}\label{sec:intro}

We investigate sharp-interface variational models for polycrystalline materials that fully respect the symmetry of the underlying lattice.
The energy is defined on partitions of a domain, each subset -- representing an individual grain -- being endowed with  an orientation modeled as an element of the quotient space $\OnG$. Here, $\G\subset\On$ is a finite point group describing the symmetries of a reference lattice, and the local orientation is measured relative to this reference.
The total energy is given by an integral over the interfaces separating adjacent grains.
The
integrand, or surface energy density,  depends on the orientations of the two neighboring grains and on the unit normal to the interface.
To ensure that the class of piecewise-constant orientation fields is closed, we assume the
energy density grows superlinearly as the misorientation angle tends to zero.

A distinctive feature of the model is that the orientation variable takes values in a manifold rather than a linear space. Classical analytical tools such as regularization, approximation, and interpolation are therefore not directly applicable, and a dedicated framework for manifold-valued fields is required.

Models of this type are widely used to study
polycrystalline  structures in metals and alloys, see, for instance,
\cite{CahnTaylor2004,BerkelsRaetzRumpfVoigt2008,ElsWir,RunnelsBeyerleinContiOrtiz2016a} and references therein.
Dependence on both orientations and lattice symmetry is likewise central to related problems such as polygonal partitions
with prescribed area and minimal perimeter \cite{DelNinThesis}, and optimal  quantization or location of measures \cite{BouPeRo,CaGoIa}.
The point-group symmetry also enters variational models of
fractional defects, see for example \cite{BadCicDeLPon18,BadCic,GolMerMil}.
Analytical aspects such as semicontinuity and relaxation -- though without an explicit treatment of the symmetry group $\G$ -- have been addressed in
\cite{ambrosio1990functionals,morgan1997lowersemicontinuity,Caraballo2008}.

The principal objective of this paper is to develop a diffuse-interface, or phase-field,  approximation of the sharp-interface model described above. Specifically,  we construct a
 phase-field approximation for surface energies with densities that vanish as the misorientation between grains tends to zero, but do so in a superlinear manner. Our approach builds on the model introduced for cohesive fracture in \cite{CFI14,ConFocIur25}.
The analysis relies on
two key abstract ingredients.
First,  compactness properties
are formulated within the framework of
metric-space-valued generalizations of Sobolev spaces and spaces of functions of bounded variation, as developed in \cite{Resh,Amb90SBVX}. Second, the construction of recovery sequences is based on a lifting theorem for $\BV$ functions taking values in  compact manifolds, established in  \cite{CanOrl20JFA}
(here we provide a simpler proof valid for our specific case, see Theorem~\ref{theoliftingmanifold}). These tools allow us to rigorously prove
the convergence of the phase-field approximation to the sharp-interface limit.

To describe the sharp-interface model in more detail, let  $\{E_i\}_i$ be
a partition
of a domain $\Omega\subset\R^N$, representing the material body.
Each set $E_i$
corresponds to an individual grain with assigned orientation $R_i\in\On$,
yielding a piecewise-constant function $\beta=\sum_i\chi_{E_i}R_i$  as the order parameter.
The grain boundaries  are contained in the discontinuity set $J_\beta$ of
 $\beta$ and the corresponding energy is given by
\begin{equation}\label{eq-sharp}
	\int_{J_\beta} \Psi(R^-, R^+,\nu_\beta)\,  \di  \mathcal{H}^{N-1}= \sum_{i\neq j} \mathcal{H}^{N-1}(\partial^* E_i\cap\partial^* E_j)\Psi (R_i,R_j,\nu_{ij}),
\end{equation}
where $\Psi:\On\times\On\times \SN  \to[0,\infty)$ is the surface energy density, $\nu_\beta$ denotes
the measure-theoretic normal to $J_\beta$, and
$\nu_{ij}$ is the normal to the intersection of the reduced boundaries
$\partial^* E_i\cap\partial^* E_j$.
A natural function space for such
energies is the space of  Special functions with Bounded Variation with values in $\On$
and vanishing absolutely continuous part of the gradient,
 \begin{equation}\label{def:SBV0}
	\PC(\Omega; \On):=\{\beta \in \SBV(\Omega; \On)\colon \nabla \beta=0\}.
\end{equation}
For the variational problems associated with \eqref{eq-sharp}  to be  well-posed in this space, certain conditions on $\Psi$ are required. In particular,
lower semicontinuity of the energy follows from
the $\BV$-ellipticity of $\Psi$ (see \cite[Theorem~5.14]{AmbFusPal00}). Additionally, it is essential that $\PC(\Omega; \On)$
is sequentially closed under the relevant convergence. This requires that sequences with bounded energy cannot converge to fields with diffuse gradient. A sufficient condition is that
$\Psi$ has infinite slope at zero,
$$
\lim_{|R^+- R^-|\to 0}\frac{\Psi(R^+, R^-,\nu)}{|R^+- R^-|}= +\infty
$$
uniformly in $\nu$,
see
\cite[Theorem~4.7]{AmbFusPal00}.
 Denoting by $\theta=|R^+- R^-|$ the relative mismatch between neighboring grain orientations, the specific growth law $\theta |\log\theta|$ of $\Psi$ for small $\theta$ was first identified
 by Read and Shockley
  in the '50s \cite{RS}.
In two dimensions, this scaling has been recently rigorously derived  as the limit of a model that accounts for the elastic distortion induced by incompatibilities and crystal defects at the boundary, see \cite{LauLuc,ForGarSpa24} and
\cite{fanzon2019derivation}.

Since $\beta$  is interpreted as a field taking values in $\On$, invariance of $\Psi$ under the action of $\G$ must be enforced. In particular,
boundaries
between regions whose orientations agree up to multiplication by  an element of $\G$
are spurious, as they
represent the same physical state, and do not contribute to the energy. Accordingly, we require
$\Psi(R, GR,\nu)=0$ for all $R\in\On$, $G\in \G$, $\nu \in \SN$, and, more generally,
\begin{equation}\label{eqintroPsiinvariance}
\Psi(GR, G'R',\nu)=\Psi(R,R',\nu) \hskip5mm\text{ for all
$R$, $R'\in\On$,
$G$, $G'\in \G$, $\nu \in \SN$. }
\end{equation}
It is important to note that, even in simple examples, this invariance cannot be removed by prescribing a consistent choice of the representative for each orientation class. See
Figure~\ref{fig-intro} for an illustration.

\begin{figure}
 \begin{center}
  \includegraphics[width=12cm]{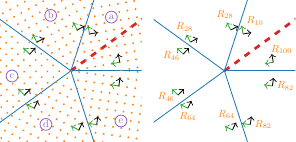}
 \end{center}
\caption{Sketch of a possible configuration. Left: the dots represent the local lattice orientation in a polycrystal with five grains, labeled $(a)$, $(b)$, $(c)$, $(d)$, $(e)$. Around each interface, a possible orientation of the two basis vectors is illustrated; each interface has a change of orientation of 18 degrees. However, these changes accumulate, and if one starts with a given orientation (interface $(a)-(b)$)
then the orientation after five interfaces differs from the initial one by 90 degrees, leading to a fictitious interface, for istance across the dashed red line in grain $(a)$. On the right, the explicit rotation matrices are indicated (angles in degrees).
}
\label{fig-intro}
\end{figure}

The goal of this paper is to construct variational phase-field models which converge to sharp-interface energies of the form \eqref{eq-sharp}. Our approximation respects the invariance
under the lattice point-group $\G$, and reproduces the Read-Shockley superlinear scaling $\theta|\log\theta|$ for small misorientation angles. The approach is motivated and inspired by the algorithms used in  \cite{BerkelsRaetzRumpfVoigt2008,ElsWir} for the segmentation of polycrystalline images.

Phase-field methods are widely employed in numerical simulations of
sharp-interface energies, and, more in general, of
free-discontinuity problems. By replacing discontinuity sets with narrow diffuse layers, they avoid the need to track interfaces explicitly and permit to use standard finite-element discretizations designed for  Sobolev functions. Classical applications include image reconstruction \cite{chambolle1995image,bourdin2000numerical,BFM} and fracture mechanics \cite{FraMar, DelPiero99, CarAmbAleDeL20, FreIur,WuNguyenNguyenChiSutula2020phase,LammenContiMosler2023}.

This strategy originated with the Ambrosio-Tortorelli approximation of the Mumford-Shah functional \cite{AmbTor90}, in which the surface energy density depends solely on the measure of the jump set, and not on the jump amplitude.
More recently, a number of phase-field approximations for cohesive fracture and segmentation energies -- where the surface energy may depend on the opening -- have been rigorously developed, see for instance
\cite{DMIur13, CFI14, CicFocZep21, CFI22, ConFocIur25}.

We approximate sharp-interface energies with superlinear
density $\Psi$ following the approach introduced in \cite{CFI14, ConFocIur25}, with the key additional challenge of incorporating the group invariance, which we describe in more detail below.
Two further technical distinctions are worth noting: first, to recover the Read-Shockley scaling law, we adapt the construction from \cite[Section~7]{CFI14}, \cite{ConFocIur25}, originally designed for general superlinear power-law growth at small openings. Second, we introduce a penalization term -- similar to that in
 \cite{CicFocZep21} for
the approximation of piecewise rigid motions with brittle interface energy -- to ensure that the diffuse gradient vanishes in the limit.

\subsection{Formulation of the main results}\label{sec:mainresults}
Let $\Omega\subset \R^N$, $N\geq 1$,   be a bounded connected set with Lipschitz boundary, and let $\G\subset\On$, for $d\geq2$, be a finite group, possibly with $d\ne N$.
 We define the space of admissible order parameters that encode the lattice symmetry as
 \begin{equation}\label{SBVG}
	\SBVG(\Omega;\Mno):=\{ \beta \in \mathrm{SBV}^2(\Omega; \Mno) \colon \beta^+(x) \in \G \, \beta^-(x) \, \text{ for }\hN\text{-a.e. } x \in J_\beta\},
\end{equation}
i.e., the set of all $\SBV$ 
functions such that $\calH^{N-1}(J_\beta)<\infty$, $\nabla\beta\in L^2$, and the
traces on the jump set differ only by multiplication by an element of $\G$ (we will sometimes say simply  ``jump in $\G$'' for brevity).

For $\e>0$ and an open set $A\subset\Omega$  we consider the energy
\begin{equation}\label{2611231254}
	\Fr_\varepsilon(\beta,v;A):= 
	\begin{cases}
		\displaystyle
		\int_{{A} } f_\varepsilon^2(v) \,|\nabla \beta|^2 \dx + \AT_\e(v;A)  &  \text{in}\ \SBVG(\Omega;\Mno) \times  H^1(\Omega;[0,1]), \\
		+\infty & \text{otherwise in}\  L^1(\Omega; {\Mno})\times L^1(\Omega),
	\end{cases}
\end{equation}
where $\AT_\e$ is a
Modica-Mortola type energy and is given by
\begin{equation}\label{AT}
	\AT_\e(v;A):= \int_{A} \left(\frac{(1-v)^2}{4\varepsilon}+\varepsilon |\nabla v|^2 \right) \dx,  \qquad v\in H^1(\Omega;[0,1]).
\end{equation}
We notice that $\Fr_\varepsilon$ is finite on the space $\SBVG$, and jumps in $\G$ of the order parameter $\beta$ do not contribute to the energy.
As in the models proposed in \cite{CFI14, ConFocIur25}, the choice of the prefactor $f_\varepsilon$ in front of $|\nabla\beta|^2$, the so-called damage coefficient, is crucial to capture the desired behavior  of the limiting surface energy density. At variance with \cite{CFI14, ConFocIur25}, and following \cite{CicFocZep21}, we allow $f_\varepsilon$ to take large values, enforcing that the limit fields have no diffuse gradient.
More precisely, we set
\begin{equation}\label{0105231047}
	f_\varepsilon(s):= M_\varepsilon \wedge \sqrt{\varepsilon} f(s) \text{ for }  s \in [0,1[, \quad f_\varepsilon(1)=M_\varepsilon, \qquad \lim_{\eps \to 0^+}M_\eps=+\infty,
\end{equation}
where  $f \in C([0,1[; [0,+\infty[)$ is a nondecreasing function satisfying  $f^{-1}(0)=\{0\}$, and
\begin{equation}\label{0405230818}
	\ell:=\lim_{s\to 1^-} \frac{1-s}{|\log(1-s)|}f(s)>0, 
\end{equation}
as for example
\begin{equation*}
	 f(s)=\ell \frac{|\log(1-s)|}{1-s}.
\end{equation*}
The singularity of $f$  at $1$ induces the logarithmic superlinear behaviour of the surface tension
at small openings (cf. Proposition \ref{prop:proprietag}), while the truncation by the diverging sequence $M_\e$  ensures the condition $\nabla \beta=0$ in the domain of the $\Gamma$-limit (see Theorem~\ref{t:comp1} and Theorem~\ref{t:comp2}).

Finally, we introduce a penalization term that asymptotically
 forces the limit fields  to take values  in
$\On$. To this end, we consider an additional positive parameter $\delta_\e\to 0^+$ and define
the modified energy
\begin{equation}\label{2611231323}
	\F_\varepsilon(\beta,v;A):=
	\Fr_\varepsilon(\beta,v;A) + \frac{1}{\delta_\varepsilon} \int_A \dist^2(\beta, \On) \dx.
\end{equation}
We shall show that
limiting behavior of the energies $\F_\varepsilon$ depends only on the asymptotic ratio
\begin{equation}\label{1002241910}
	\lambda:=\lim_{\varepsilon\to 0}\frac{\varepsilon}{\delta_\varepsilon}\in [0,+\infty],
\end{equation}
that we assume to exist.
In view of the invariance of the energy $\F_\e$ under the action of the symmetry group $\G$,  it is natural to adopt the following notion of convergence modulo $\G$.
A family $(\beta_\varepsilon)_\varepsilon\subset L^1(\Omega; \Mno)$ is said to converge to $\beta$ in the $\LG$-sense if
 there exists a sequence $G_\varepsilon\in L^\infty(\Omega; \G)$ such that
\begin{equation}\label{0602242311}
	G_\varepsilon \beta_\varepsilon \to \beta \quad \text{in }L^1(\Omega; \Mno).
\end{equation}
With respect to this notion of convergence,  we prove a compactness result for sequences with equi-bounded energy and  characterize the $\Gamma$-limit of $\F_\e$. More precisely,
for any $\lambda \in [0,+\infty]$ we define the sharp-interface functional
\begin{equation}\label{f:fzero}
	\F^\lambda (\beta,v):=
	\begin{cases}\displaystyle 
		\int_{J_\beta} g_\lambda( \beta^-,\beta^+)\, \dhN & \text{if } \beta\in \SBV_0(\Omega;\On) ,\ v=1,\\
		+\infty  & \text{otherwise in}\   L^1(\Omega; {\Mno})\times L^1(\Omega),
	\end{cases}
\end{equation}
where the surface energy density $g_\lambda$ is determined via an asymptotic cell problem that will be introduced and analyzed in detail in Section \ref{sec:surfen}.

Due to the invariance of the energy $\F_\e$ under the action of $\G$, the limit density $g_\lambda$ inherits the same invariance property. That is, for every $R^+,R^-\in \On$ we have
\begin{equation}
	\label{eq-glambda-invariance}
	g_\lambda(R^+,R^-)=g_\lambda(GR^+,G'R^-)\qquad \forall \ G,G'\in \G,
\end{equation}
(see Proposition \ref{p:bpg}).
In particular, the functional $\F^\lambda$ depends only on the equivalence classes of $\beta$ in the quotient space  $\OnG$, namely,
\begin{equation}\label{eq-limit-classi}
	\F^\lambda(\beta,v)=\F^\lambda(G\beta,v) \quad\text{for every }  (\beta,v,G) \in L^1(\Omega; {\Mno})\times L^1(\Omega)\times  \BV(\Omega;\G).
\end{equation}

The main result of this paper is the following.
\begin{theorem}\label{thm:main}
	The following properties hold:
	\begin{itemize}
		\item[(C)] 
		For any $(\beta_\varepsilon, v_\varepsilon)_\varepsilon \subset L^1(\Omega;\Mno \times[0,1])$ satisfying 
		\begin{equation}\label{compest}
			\limsup_{\varepsilon\to 0} \F_\varepsilon(\beta_\varepsilon, v_\varepsilon)<+\infty
		\end{equation}
		there exists $\beta \in \PC(\Omega; \On)$ such that, up to a subsequence,
		\begin{equation}\label{comp}
			\beta_\varepsilon \to \beta \text{ in }\LG(\Omega; \Mno),\qquad v_\varepsilon \to 1 \text{ in }L^1(\Omega);
		\end{equation}
		\item[(LB)] 
		For any $(\beta_\varepsilon, v_\varepsilon)_\varepsilon \subset L^1(\Omega;\Mno \times[0,1])$ and $\beta \in \PC(\Omega; \On)$ satisfying \eqref{comp}, it holds
		\begin{equation*}\label{liminf}
			\F^\lambda(\beta,1)\leq \liminf_{\varepsilon \to 0} \F_\varepsilon(\beta_\varepsilon, v_\varepsilon);
		\end{equation*}
		\item[(UB)] 
		For any $\beta \in \PC(\Omega; \On)$ there exists $(\beta_\varepsilon, v_\varepsilon)_\varepsilon \subset L^1(\Omega;\Mno \times[0,1])$ satisfying \eqref{comp} and
		\begin{equation}\label{limsupthm}
			\F^\lambda(\beta,1)\geq \limsup_{\varepsilon \to 0} \F_\varepsilon(\beta_\varepsilon, v_\varepsilon).
		\end{equation}
	\end{itemize}
\end{theorem}

\subsection{Application to image segmentation}
 To apply Theorem~\ref{thm:main} to the
 segmentation of images of polycrystals,
 we assume $N=d$ and
 introduce a fidelity term inspired by \cite{BerkelsRaetzRumpfVoigt2008,ElsWir}.
Let $w\in C(\Omega; \R)$
represent the grayscale image of a polycrystal, and let $\sigma>0$ be a small parameter of the order of the lattice spacing.
We fix a finite number of  lattice vectors
$v_1,\,\dots,\, v_K \in \Rn$, with $|v_k|\leq \sigma$,
and positive weights $\alpha_k>0$, $k=1,\ldots,K$.

For a given domain $\Omega'$ such that $\{x\colon \dist(x,\Omega') <2\sigma\}\subset\subset \Omega$
we define the fidelity term by
\begin{equation}\label{0212231710}
	\FT(\beta; \Omega'):= \int_{E_\beta \cap \Omega'} \sum_{k=1}^{K}\alpha_k\sum_{G\in\G} (w(x+ \big(\beta(x)\big)^{-1}Gv_k)- w(x))^2 \dx,
\end{equation}
where
\begin{equation*}
	E_\beta:=\{x \in \Omega \colon \det(\beta(x))\ne 0, \, \|\big(\beta(x)\big)^{-1}\|_\mathrm{op}\leq 2\},
\end{equation*}
and
$\|\cdot\|_\mathrm{op}$ denotes the operator norm (i.e., the largest singular value). This definition  ensures that the points $x$ included in the integral satisfy $x+\big(\beta(x)\big)^{-1}Gv_k\in \Omega$, so that $\FT$ is well-defined.
Since the vectors $v_k$ belong to the lattice and $\G$ is the associated point group, the vectors $Gv_k$ also lie in the lattice. Thus, if $w$ is an image of a lattice with orientation $R\in\On$, then $w(x+RGv_k)=w(x)$  for all $G$ and $k$, and the fidelity term $\FT$ is minimized for $\beta=R^T$.

The definition of the fidelity term $\FT$ ensures invariance under the action of the group $\G$. In particular,
 for every $\beta\in \SBVG(\Omega; \Mno)$ and every $G\in L^\infty(\Omega;\G)$ it holds that
\begin{equation}\label{eq-fidelity-invariance}
	\FT(G\beta;\Omega')=\FT(\beta;\Omega').
\end{equation}
Moreover, by the continuity of $w$
and the fact that any sequence with bounded energy converges to a limit valued in $\On$, which implies $|\Omega\setminus E_{\beta_\eps}|\to0$,
the fidelity term $\FT$
 is continuous
with respect to the $\LG$ convergence of $\beta$ along sequences with bounded energy. 
Hence, by Theorem~\ref{thm:main},
for any fixed $\gamma>0$ the sequence of functionals
\begin{equation}\label{2611231320}	\cH_\varepsilon(\beta,v):=\F_\varepsilon(\beta,v;\Omega) + \gamma \, \FT(\beta; \Omega')
\end{equation}
$\Gamma$-converges
(with respect to the $\LG\times L^1$ convergence of  $(\beta,v)$)
to
\begin{equation}\label{2611231320l}
\cH^\lambda \EEE(\beta,v):=\F^\lambda(\beta, v \EEE;\Omega) + \gamma \, \FT(\beta; \Omega').
\end{equation}
In view of the equi-coerciveness established in Theorem \ref{thm:main}, this also implies convergence of minima and almost minimisers of $\cH_\varepsilon$ to those of $\cH^\lambda$.
Clearly the same would also apply to other types of fidelity terms,  as far as they have the appropriate invariance under the action of the group $\G$ and they are continuous with respect to the $\LG$ convergence along sequences with bounded energy.

\subsection{Main difficulties and novelties}

The key novelty of this paper lies in the presence of the point group $\G$, under whose action the sharp interface model remains invariant. Crucially, we are able to incorporate this invariance  also at the level of the approximating functionals. This requires working with an \emph{ad hoc}
notion of convergence for the fields that respects the group action, in appropriate spaces.

 A fundamental preliminary challenge is the choice of
a suitable functional space for the order parameters.
Two main approaches can be considered, both in the formulation of the statement and in the proofs. In the presentation above, we opted for the space $\SBVG$ defined in \eqref{SBVG},
which is a subset of the classical space $\SBV$.
A version of this space (under different notation)
appears naturally also in the modeling of fractional dislocations
in two dimensions \cite{BadCic, GolMerMil}.

An alternative approach is to define the functionals directly on functions taking values in the quotient space $\Mno/\G$. In this framework, one can use the metric structure of $\Mno/\G$ and work in the more abstract framework of Sobolev and $\BV$ spaces of functions with values in  metric spaces, as introduced by Ambrosio in \cite{Amb90SBVX} (see also \cite{Resh}).  Remarkably, in this metric space approach, the approximating functionals are defined on Sobolev functions as usual in phase field approximation, in particular the fields belong to the space $H^1(\Omega; \Mno/\G)$. However,
this more abstract setting comes with significant difficulties.
The lack of linear structure in the quotient space restricts the applicability  of  many
standard analytical and numerical tools, such as the Finite Elements Method.
This poses substantial obstacles for both theoretical analysis and numerical simulations.

Our chosen approach thus balances between capturing the essential group invariance and maintaining a workable structure for analysis and approximation.

In order to compare the two approaches, it is important to clarify the relation between $\SBVG(\Omega; \Mno)$ and $H^1(\Omega; \Mno/\G)$. 
It is  plain to see that for any function in $\SBVG(\Omega; \Mno)$, by passing pointwise to equivalence classes, one obtains a function in $H^1(\Omega; \Mno/\G)$ (cf.\ Remark~\ref{r:sbvgvsh1}). Conversely, it is not true in dimension $d\geq 2$ that any function in $H^1(\Omega; \Mno/\G)$ has a (pointwise) representative in $\SBVG(\Omega; \Mno)$, since in general the reconstructed jump set may have infinite measure, as in the counterexample of Lemma~\ref{le:contres}. 
Therefore, in general, the approximating functionals defined in $H^1(\Omega; \Mno/\G)$ are smaller than the ones defined in 
$\SBVG(\Omega; \Mno)$, cf.\ \eqref{0904251306}.

When the reference domain $\Omega$ has dimension 1, instead, its total ordering 
allows to prove that the two spaces are equivalent, in the sense described just above, and suitably represented 
by $H^1$ functions with values in $\Mno$, see Lemma~\ref{l:clean}.
Further, a similar property holds if the target set of the functions is restricted, in particular for the domains of the limiting functionals, which are
the spaces $\SBVG(\Omega; \On)$ and  $H^1(\Omega; \On/\G)$: in fact, in this case one may work in the framework of Sobolev or $\BV$ functions valued into manifolds, and 
a lifting is possible, see Theorem~\ref{theoliftingmanifold} below.

Basing on the former relations,
we develop the two approaches in such a way that each one can be followed independently, in particular Theorem~\ref{thm:main} can be proven without using the theory 
 of Sobolev or $\BV$ functions valued in metric spaces, as we now describe.
 
As for Compactness (C), we employ the convergence of the fields $\beta_n$ with equibounded energy to a function with values in $\On$ together with the fact that the functions $\mathrm{det}\,\beta_n$ are equibounded
 in $W^{1,1}$  in $\{v_n\geq c\}$ for $c>0$. The latter 
follows from the $\G$-invariance of the determinant
 and the uniform integral bound for $\nabla \beta_n$ in $\{v_n\geq c\}$.
Then, by the lifting result in Theorem~\ref{theoliftingmanifold}, we obtain a uniform bound in
 $\BV$, up to modifications in
 sets with bounded perimeter. 
 
 Both in the characterization of the limit domain $\PC$ and in the Lower Bound (LB) one may 
 reduce to one dimensional slices so,  by Lemma~\ref{l:clean},
 work with standard Sobolev function and adapt the analogous properties
from \cite{CFI14}. In the characterization of the limit domain we use here a strategy similar to the one in \cite{ConFocIur25}.
The Upper Bound (UB) employs a density result in $\PC$ through functions with jump set given by a finite union of pieces of hyperplanes \cite{BCG14}, allowing to use, where the jump is flat, optimal profiles from the one dimensional cell formula for $g_\lambda$, while a careful construction applies to deal with
neighborhoods of the $N-2$ dimensional intersections of different faces. 
 
Taking advantage of the abstract theory in \cite{Amb90SBVX}, one may prove the more general Compactness (C)' in Theorem~\ref{thm:maineq} (recall \eqref{0904251306}) arguing analogously to the ``classical'' $\BV$ case in \cite{CFI14}, while the more general Lower Bound (LB)' follows similarly to (LB), due to the 1d equivalence from Lemma~\ref{l:clean} and some fine properties from \cite{Amb90SBVX}. Conversely, the strongest Upper Bound is (UB), so (UB)' follows.  

The paper is structured as follows. In Section~\ref{sec-sobolev-metric} we introduce some notation and recall the basic theory of Sobolev and $\BV$ functions valued in metric spaces, discuss the relation with the counterparts which are subspaces of $\SBVG$, and introduce the formulation for functions valued in equivalence classes along with the main result Theorem~\ref{thm:maineq}. 
Then we prove in Section~\ref{subsec:comp} compactness  and the characterization of the domain for limiting fields.  Section~\ref{sec:surfen} is devoted to  define, describe, and characterize  the surface tension.  Such analysis will be crucial 
to prove the lower bound and the upper bound, in Section~\ref{sec:proofmain}.

\section{Formulation with Sobolev functions  valued in metric spaces}\label{sec-sobolev-metric}

This Section 
addresses the formulation of the problem with Sobolev functions valued in metric spaces.
We start with some technical preliminaries and notation, which are relevant for the entire paper.

\subsection{Notation}
In the entire paper, $\Omega$ is a bounded connected open subset of $\R^N$, $N\geq 1$,  with Lipschitz boundary, and $\mathcal{A}(\Omega)$ and $\mathcal{B}(\Omega)$ denote the set of open and  Borel \EEE subsets of $\Omega$, respectively.
When $\Omega$ is one dimensional we often denote it with $I$.
The Lebesgue measure and the $k$--dimensional  Hausdorff \EEE measure in $\R^N$ are denoted by $\calL^N$ and $\calH^{k}$, respectively.
We denote by $\mathrm{M}^{d\times d}$ the space of $d\times d$ matrices, with $d\geq 2$.
For $\beta\in\mathrm{M}^{d\times d}$, 
$\beta^T$ and $|\beta|$ are respectively its transpose matrix and  its Euclidean norm.

Given  a finite subgroup $\G$ of $\On$, that will be fixed along the paper,
we denote the quotient space of $\Mno$ with respect to $\G$ by  $\Mno/\G$ \EEE
and its elements by
\[
[\beta]=\G\beta:=\{G\beta\colon  G\in \G\}, \qquad \beta\in \Mno.
\]
The space $\Mno/\G$ is a separable complete metric space endowed with the metric
\begin{equation}
	\label{eq-metric-X}
	\dG([\beta], [\widetilde\beta]):= \min_{G\in \G} |\beta-G\widetilde\beta|\qquad \text{for every } [\beta], [\widetilde\beta]\in \Mno/\G.
\end{equation}

\subsection{Functional setting I: Metric-space valued Sobolev functions}\label{subsec:functionalsetting}
In what follows we always assume that
\begin{equation*}
	(X,\dix) \text{ is a separable locally compact metric space},
\end{equation*}
having in mind the application to $X=\Mno/\G$ or $X=\OnG$.
The latter is here seen as a subset of the first, endowed with the same metric, as defined in \eqref{eq-metric-X}. We observe that $\OnG$ could also be seen as a manifold, then it is natural to adopt on each tangent space the metric
that $\On$ inherits from $\Mno$, or $\On/\G$ from $\Mno/\G$, and then extend it to each connected component using geodesics; the result is a metric different from, but equivalent to, the one in
\eqref{eq-metric-X}.
The property of  $X$ being locally compact will be needed for the definition of the space  $\BV(\Omega; X)$, and not for Sobolev  spaces.  

\smallskip

For every $p\in[1,+\infty]$, we define the space $L^p(\Omega  ;  X)$  as  the space of all Borel functions $u:\Omega\to X$ such that there exists a $z_0\in X$ for which  $\dix (u(\cdot),z_0):\Omega\to\R$ is in $L^p(\Omega)$.
\begin{remark}\label{rem:1104251642}
	The $\LG$--convergence corresponds to the convergence of the equivalence classes in $L^1(\Omega;\Mno/\G)$, namely
	\[
	\beta_\varepsilon\to \beta \text{ in }\LG(\Omega;\Mno) \quad \text{if and only if} \quad [\beta_\varepsilon] \to \ebeta \text{ in }L^1(\Omega; {\Mno}/{\G}),
	\] 
	which is also equivalent to
	\[
	\lim_{\e \to 0} \int_\Omega \dG([\beta_\e(x)],[\beta(x)])\dx= 0.
	\]
\end{remark}

Let us consider a function $\beta\in \SBVG$, as defined in
\eqref{SBVG}. Then $\nabla\beta\in L^2$,
and $\beta^+\in \G \beta^-$ at almost every jump point. Therefore if we consider the equivalence classes corresponding to $\beta$, defined by
$[\beta]:\Omega\to\Mno/\G$, $x\mapsto [\beta(x)]$, the jumps disappear. In particular, this function is expected to be in $H^1(\Omega;\Mno/\G)$ (which has not been defined yet).

To formalize this setting, we recall the notion of Sobolev functions with values in a metric space from \cite{Resh}.
By $\Lip_1(X;\R^m)$ we denote the set of Lipschitz functions from  $X$  to $\R^m$ with Lipschitz constant $\Lip(\varphi)$ not larger than 1, 
for $m=1$ we write $\Lip_1(X)$.

\begin{definition} \label{def-sobolev-metric}
	The space $W^{1,p}(\Omega;X)$, $p\geq 1$, consists of those functions $u \in L^p(\Omega;X)$ such that
$\varphi\circ u \in W^{1,p}(\Omega)$ for every $\varphi\in \Lip(X)$.
For $m\in\N_{>0}$, if we set
\begin{equation}\label{f:ReshUG-m1}
		|\nabla u|_{(m)}:={\rm ess\, sup}\{|\nabla(\varphi\circ u)| : \   \varphi \in \Lip_1(X;\R^m) \},
	\end{equation}
	it holds $|\nabla u|_{(m)}\in L^p(\Omega)$.
We shall mostly take $m=d^2$, and when the meaning is clear from the context write briefly $|\nabla u|:=|\nabla u|_{(d^2)}$.
	In dimension $N=1$ we use the notation $|u'|$ for $|\nabla u|_{(1)}$.
\end{definition}

\begin{remark}
\label{rem-dep-m}
We observe that  for any $m\in\N_{>0}$
\begin{equation}\label{eqboundsnablaumnablau1}
 |\nabla u|_{(1)}\le
 |\nabla u|_{(m)}\le
 \sqrt {N\wedge m}\, |\nabla u|_{(1)} \hskip5mm\text{ pointwise almost everywhere.}
\end{equation}
To prove the first inequality we extend any function in $\Lip_1(X)$ to a function in $\Lip_1(X;\R^m)$ adding $m-1$ components which are identically zero. To prove the $\sqrt m$ bound in the second one, we
observe that, if $\varphi=(\varphi_1,\ldots,\varphi_m)\in \Lip_1(X;\R^m)$, then for each $i\in\{1,\dots, m\}$ we have
$\varphi_i\in \Lip_1(X)$, 
and $|\nabla (\varphi\circ u)|=|(\nabla (\varphi_1\circ u), \dots, \nabla(\varphi_m\circ u))|
\le \sqrt m \max_i |\nabla (\varphi_i\circ u)|$.
To prove the final inequality,
let $\varphi\in \Lip_1(X;\R^m)$, fix $v\in \mathbb{S}^{N-1}$, and observe that
for almost every $x\in\Omega$ and $t$ small enough, writing $I_t$ the segment with endpoints $0$ and $t$, one has
\begin{equation}
\begin{split}
 |\varphi(u(x+tv))-\varphi(u(x))|
 \le & |\dX(u(x+tv), u(x))|
 \le \int_{I_t} |\nabla (\dX\circ u)|(x+sv) \ds\\
 \le& \int_{I_t} |\nabla u|_{(1)}(x+sv)\ds.
\end{split}\end{equation}
Therefore if $x$ is a Lebesgue point of  the slice of $|\nabla u|_{(1)}$,
then
\begin{equation}
\limsup_{t\to0} \frac{|\varphi(u(x+tv))-\varphi(u(x))|}{|t|}
\le \limsup_{t\to0} \frac{1}{|t|}
\int_{I_t} |\nabla u|_{(1)}(x+sv)\ds
= |\nabla u|_{(1)}(x)
\end{equation}
and we obtain $|\nabla_v (\varphi\circ u)|\le |\nabla u|_{(1)}$ almost everywhere.
Finally, for any $g\in W^{1,p}(\Omega;\R^m)$ we have
\begin{equation}
 |\nabla g|^2(x)
 =\sum_i |\partial_i g|^2(x)
 \le N \max_i |\partial_i g|^2(x).
\end{equation}
Using this with $g=\varphi\circ u$,
we obtain
$|\nabla (\varphi\circ u)|_{(m)}\le \sqrt N |\nabla u|_{(1)}$ almost everywhere which concludes the proof of \eqref{eqboundsnablaumnablau1}.

We next show that
\begin{equation}\label{eqnablaumNm}
\text{if $X\subset\R^m$, then $|\nabla u|_{(m)}=|\nabla u|_{\R^{m\times N}}$,}
\end{equation}
 where the last norm is the usual Euclidean norm of the weak gradient.
 To see this, observe that any $\varphi\in \Lip_1(X;\R^m)$ has an extension $\widehat\varphi\in \Lip_1(\R^m;\R^m)$, and by the chain rule
 $|\nabla (\widehat\varphi\circ u)|=|(\nabla\widehat\varphi)\circ u \, \nabla u|\le |\nabla u|$.
At the same time, for $\varphi(x)=x$ the inequality is an equality, which proves \eqref{eqnablaumNm}.
A similar argument shows that the bounds in
\eqref{eqboundsnablaumnablau1} are optimal.
\end{remark}

\begin{remark}\label{r:sbvgvsh1}
	The functions in $\SBVG(\Omega;\Mno)$ are stricty related to the functions in the Sobolev space 
	$H^1(\Omega;\Mno/\G):=W^{1,2}(\Omega;\Mno/\G)$.
To see this,
	for a Borel function $\beta\colon \Omega \to \Mno$ we define
	\[
	[\beta]\colon \Omega \to \Mno/\G, \qquad [\beta](x):=[\beta(x)], \  x\in \Omega.
	\]
By the chain rule in $\BV$, if $\beta \in \SBVG(\Omega;\Mno)$ and $\varphi\in \Lip( {\Mno}/{\G};\R^m)$, then
	$\varphi\circ \ebeta$ belongs to $H^1(\Omega;\R^m)$, and
	\begin{equation}\label{f:ReshUG1}
		|\nabla(\varphi\circ \ebeta)| \leq \Lip(\varphi)  |\nabla\beta|, \qquad  \hbox{a.e.\ in }\ \Omega,
	\end{equation}
	so that $\ebeta \in H^1(\Omega; \Mno/\G)$, and from \eqref{f:ReshUG1}
	\begin{equation}\label{f:ststgrad}
		|\nabla \ebeta|_{(m)}\leq |\nabla\beta|, \quad m\in \N_{>0}.
	\end{equation}
	Hence the functions in $\SBVG$ can be identified with Sobolev functions with values in the quotient space $\Mno/\G$.
\end{remark}

If the domain is one dimensional, the two spaces $\SBVG(I;\Mno)$ and $H^1(I;\Mno/\G)$ are actually equivalent. We can moreover ``clean the jumps'' of a $\SBVG(I;\Mno)$ function, obtaining a representative in $H^1(I;\Mno)$, as shown in Lemma~\ref{l:clean} below. This result will be crucial for our proof.

If the domain is at least two dimensional, the two spaces are not equivalent, see Lemma~\ref{le:contres} below. However, the counterexample is based on the fact that functions in $\SBVG(I;\Mno)$ are required to have jump set of finite measure. If this property is eliminated from the definition, then we expect that it is possible to identify the two spaces
(see Remark~\ref{ref:lifting} below).

\begin{lemma}\label{l:clean}
Let $I=(a,b)\subset \R$ be a bounded interval, and let $\beta: I\to \Mno$ be a Borel function. Then the following hold:
\begin{itemize}
\item[i)] if  $[\beta] \in H^1(I;\Mno/\G)$, then there exists $G_\beta \in L^\infty(I;\G)$ such that $G_\beta \beta \in H^1(I; \Mno)$,
$G_\beta \beta(a)=\beta(a)$, and
$|[\beta]'|_{(m)}=| (G_\beta\beta)'|$ almost everywhere in $I$ for any $m\in\N_{>0}$;
\item[ii)]  if, in addition, $\beta\in \SBVG(I;\Mno)$,  then $(G_\beta\beta)'=G_\beta\beta'$ almost everywhere in $I$.
\end{itemize}
\end{lemma}

\begin{proof}
Let $\beta: I\to \Mno$ be such that $[\beta](x):=[\beta(x)]$ belongs to $H^1(I;\Mno/\G)$, and therefore it admits a continuous representative (see, e.g. \cite[Remark 2.2]{Amb90SBVX}).
First of all, we prove that $[\beta]\in  C^{1/2}(I;\Mno/\G)$. To see this, fix $x\in I$, and consider
$\varphi_x(U):=\dG(U,[\beta](x))$, $U\in \Mno/\G$. Since $\varphi_x$ is 1-Lipschitz, and $[\beta] \in H^1(I;\Mno/\G)$,
we have
\begin{equation*}\begin{split}
 \dG([\beta](x),[\beta](y)])
 =&\varphi_x([\beta](y))-\varphi_x([\beta](x))
 \le \int_{[x,y]} |(\varphi_x\circ[\beta])'| \dt
 \\
 \le& |x-y|^{1/2} \|(\varphi_x\circ[\beta])'\|_{L^2(I)}
\le |x-y|^{1/2} \| \left|[\beta]'\right|_{(m)}\|_{L^2(I)}.
\end{split}\end{equation*}

In order to prove i),  we first construct a sequence $(\beta_k)_k$ of functions in $H^1(I;\mathrm{M}^{d\times d})$ such that
	\begin{equation}\label{eqclaimseqbetakh1}
		\| \beta'_k\|_{L^2(I)}^2 \le \| \left|[\beta]'\right|_{(m)}\|_{L^2(I)}^2, \hskip1cm 
		\beta_k(a)=\beta(a)
		, \hskip1cm 
		\|\dG([\beta_k],[\beta])\|_{L^\infty(I)}\to0. 
	\end{equation}
	To do this, 
	we subdivide the interval $I$ in $k$ subintervals, choosing $a_0:=a<a_1<\dots <a_k:=b$
	such that $|a_{i+1}-a_i|\le 2 (b-a)/k$ for all $i$. 
	Further, for $i\ge 1$
	the function 
$\phi_i:\mathrm{M}^{d\times d}\to\R$ defined by
$\phi_i( U):=\dG([\beta(a_{i-1})],  U)$
	is 1-Lipschitz, hence $ \phi_i\circ [\beta]\in H^1(I)$ and, 
	by \eqref{f:ReshUG-m1},
	\begin{equation}\label{f:clean1}
		\phi_i^2([\beta(a_i)])
		\le (a_i-a_{i-1})
		\int_{a_{i-1}}^{a_i} \Big|\frac{\di}{\dt} (\phi_i\circ [\beta])\Big|^2 \dt
		\le (a_i-a_{i-1})
		\int_{a_{i-1}}^{a_i} | [\beta]'|_{(m)}^2\dt.
	\end{equation}
	We set $\beta_k(a_0)=\beta(a_0)$, then inductively select $\beta_k(a_i)\in [\beta](a_i)$ 
	such that
	\begin{equation}\label{f:clean2}
		|\beta_k(a_i)-\beta_k(a_{i-1})|=
		\dG([\beta(a_i)],[\beta(a_{i-1})])=\phi_i([\beta(a_i)]),
	\end{equation}
	and define the function $\beta_k$ as the affine interpolation in each of the intervals. 
	By \eqref{f:clean2} and \eqref{f:clean1} we get
	\begin{equation}\label{f:clean3}
		\int_{a_{i-1}}^{a_i} |\beta_k'|^2 \dt
		=\frac1{a_i-a_{i-1}} |\beta_k(a_i)-\beta_k(a_{i-1})|^2\le \int_{a_{i-1}}^{a_i} | [\beta]'|_{(m)}^2\dt,
	\end{equation}
	and the first condition in \eqref{eqclaimseqbetakh1} follows. The second one holds by the choice of $\beta_k(a_0)$. The third one follows by the embedding of $H^1$ in $C^{1/2}$, with a computation similar to \eqref{f:clean1}. More explicitly, we obtain for any $t\in (a_{i-1},a_i)$, 
	\[
	\begin{split}
\dG([\beta_k](t),[\beta](t))
		& \le |\beta_k(t)-\beta_k(a_{i-1})|+
		\phi_i([\beta(t)]) \leq  	\phi_i([\beta(a_i)])+ 	\phi_i([\beta(t)])\\
		& \le 2 \sqrt{a_i-a_{i-1}}\left(\int_{a_{i-1}}^{a_i} |[\beta]'|_{(m)}^2\dt\right)^{\frac12} \leq \frac{C}{\sqrt{k}} 
	\end{split}
	\]
	and therefore the uniform convergence  to zero in  \eqref{eqclaimseqbetakh1} is proved.
	
	Now
	by compactness, one obtains that $(\beta_k)_k$ has a subsequence converging weakly in $H^1(I; \mathrm{M}^{d\times d})$ (and almost everywhere in $I$) to some 
$\beta_*\in H^1(I;\mathrm{M}^{d\times d})$, with $\beta_*(a)=\beta(a)$. By the continuity of the distance $\dG$ we obtain $[\beta_*]=[\beta]$ almost everywhere, and hence there exists $G_\beta: I\to \G$ such that $\beta_*=G_\beta\beta$ almost everywhere. Since $\G$ is finite one immediately obtains that $G_\beta$ can be chosen measurable. By the lower semicontinuity of the norm, and \eqref{f:clean3} we obtain
	\[
	\int_I |\beta_*'|^2 \dt \leq \int_I | [\beta]'|_{(m)}^2\dt= \int_I | [\beta_*]'|_{(m)}^2\dt 
	\] 
that, together with \eqref{f:ststgrad} (applied to $\beta_*$), implies that $|\beta_*'|=|[\beta]'|_{(m)}$ a.e. in $I$, concluding the proof of the first part of the Lemma.
	
	Now we assume, in addition, that $\beta$ belongs to $\SBVG(I;\Mno)$, and we show that $(G_\beta\beta)'=G_\beta \beta'$ almost everywhere in $I$. 
	To do this, fix any $t\in I$ which is a point of approximate differentiability of $\beta$ and $\beta_*$,  
	so that
	\begin{equation}\label{eqreconstrlebspt}
		\lim_{r\to0} \int_{t-r}^{t+r}
		\frac{|\beta(s)-\beta(t)-\beta'(t)(s-t)|}{r^2}
		+\frac{|\beta_*(s)-\beta_*(t)-\beta_*'(t)(s-t)|}{r^2}
		\, \ds=0.
	\end{equation}
	If $t$ is also a Lebesgue point of $G_\beta$, as $\G$ is discrete, we obtain that the set 
	\[
	E_r:=\{s\in (t-r,t+r): G_\beta(s)= G_\beta(t)\}
	\] 
	obeys $|E_r|/r\to2$ as $r\to0$. 
Using $\beta_*(s)=G_\beta(t)\beta(s)$ in $E_r$ and \eqref{eqreconstrlebspt} we obtain
	\begin{equation*}
\lim_{r\to0} \int_{E_r}
\frac{|(G_\beta(t)\beta'(t)-\beta_*'(t))(s-t)|}{r^2} \, \ds
=0
	\end{equation*}
	which implies  the desired identity
	$G_\beta(t)\beta'(t)=\beta_*'(t)$.
\end{proof}

The next Lemma shows with an example that
there are functions in
$\SBV(\Omega;\Mdue/\G)$, with $N$, $d\ge2$,
for which we cannot find
a representative with jump set of finite measure. Therefore in dimension at least two the result proved in Lemma~\ref{l:clean} does not hold.

\begin{lemma}\label{le:contres}
	If $N\ge2$, $d\ge2$, and $\G\cap\SOn\ne\{\Id\}$, then there is $u\in H^1((0,1)^N;\Mno/\G)$ such that
	there is no $\beta_0\in \SBVG((0,1)^N;\Mno)$ with $u=[\beta_0]$.
\end{lemma}
\begin{proof}
	We provide an explicit construction in the case $N=2$, the case $N>2$ follows by extending the construction so that it does not depend on the variables $x_3,\dots, x_N$, in the sense that $u^{(N)}(x)=u^{(2)}(x_1,x_2)$.
	
	Select $G\in\G\cap\SOn$, $G\ne\Id$, and then
	pick a curve $\gamma\in C^1([0,2\pi];\SO(d))$ such that
	$\gamma(0)=\Id$, $\gamma(2\pi)=G$. This exists since $\SO(d)$ is connected.
	Fix $\varphi\in C^1_c((0,1);[0,1])$ such that $\{\varphi>0\}=(1/4,3/4)$.
We start working in the unit ball $B_1$ and use polar coordinates. We define,
for $t\in (0,1)$ and $\theta\in [0,2\pi)$,
	\begin{equation*}
	\beta_{1}(t e^{i\theta}) :=
\varphi(t) \gamma (\theta).
	\end{equation*}
	We observe that $\beta_{1}\in \SBV(B_1;\Mno)$, 
	with $\beta_1=0$ around $\partial B_1$, $\nabla \beta_1\in L^\infty(B_1)$,
	$J_{\beta_1}=(1/4,3/4)e_1$,
	$\beta_1^-(te_1)=\varphi(t)G$; $\beta_1^+(te_1)=\varphi(t)\Id$, so that
	$\beta_1\in \SBVG(B_1;\Mno)$.
	We define
	\begin{equation*}
		c_2:=\|\nabla \beta_1\|_{L^2(B_1)}^2\in (0,\infty), 
		\hskip5mm
		c_J:=\calH^1(J_{\beta_1}\cap B_1)=\frac 12,
		\hskip5mm
		c_D:=\int_0^1 |\beta_1^+-\beta_1^-|(te_1) \dt\in (0,\infty).
	\end{equation*}
	 
	Consider countably many pairwise disjoint balls 
	$B_k:=B_{r_k}(x_k)$ contained in the unit square $Q:=(0,1)^2$ such that
	\begin{equation*}
		\sum_k r_k=\infty.
	\end{equation*}
	To choose them one can, for example, take any dyadic decomposition of the unit square into countably many subsquares and consider the circle inscribed in each subsquare.
	We set, for $a:\N\to(0,\infty)$ chosen below and $K\in\N$,
	\begin{equation*}
		\beta_K(x):=
		\sum_{k=1}^K a_k \chi_{B_k}(x) \beta_1 \left( \frac{x-x_k}{r_k}\right)
	\end{equation*}
	and verify that $\beta_K\in \SBVG(Q;\Mno)$,
	with $[\beta_K]\in H^1(Q;\Mno/\G)$.
	We compute
	\begin{equation*}
		\|\nabla \beta_K\|_{L^2(Q)}^2
		=c_2\sum_{k=1}^K a_k^2  ,
		\hskip5mm
		\calH^1(J_{\beta_K}\cap Q)
		=c_J\sum_{k=1}^K r_k ,
		\hskip5mm
    \int_{J_{\beta_K}} |\beta_K^+-\beta_K^-| \, d\calH^1
		=c_D\sum_{k=1}^K a_k r_k .
	\end{equation*}
	Obviously $r_k\to0$; we select $a_k=2^{-k}$. Then
	the sequence $(\beta_K)_K$ converges uniformly and strongly in $\BV$ to some $\beta_\infty$.
	In each ball $B_k$, we have $\beta_\infty=\beta_h$ for any $h\ge k$; outside the union of the balls $\beta_\infty=0$.
	One checks that $ u:=[\beta_\infty] \in H^1(Q;\Mno/\G)$.
	
	It remains to show that there is no $\beta\in \SBVG(Q;\Mno)$ with $[\beta]=u$.
	
	Assume there was one, pick any $h\in \N$, and consider the ball $B_h$. 
	We claim that necessarily 
	\begin{equation}\label{eqjumpjbetaBh}
		\calH^1(J_\beta\cap B_h)\ge \frac 12 r_h.
	\end{equation}
	If this holds, then $\calH^1(J_\beta\cap Q)\ge \frac12 \sum_h r_h=\infty$, so that $\beta\not\in \SBV^2(Q;\Mno)$  and the proof is concluded.
	
	It remains to prove \eqref{eqjumpjbetaBh}.
	Assume this was not the case. 
	Then by slicing there would be a set $E\subset(0,1)$ 
	with $\calL^1(E)>1/2$ such that for any $t\in E$ the 
	$2\pi$-periodic
	function 
	\begin{equation*}
		v(\theta):= \beta(x_h+r_ht e^{i\theta})
	\end{equation*}
	obeys $v\in W^{1,1}_\mathrm{\loc}(\R;\Mno)$, we identify it with its continuous representative.
By our definitions, $[v(\theta)]=a_h\varphi(t)[\gamma(\theta)]$.
	We pick $t\in E$ with $\varphi(t)>0$ and seek a contradiction.
	We define $w:[0,2\pi]\to\Mno$ by
	\begin{equation*}
		w(\theta):= v(\theta)(\gamma(\theta))^{-1}.
	\end{equation*}
	This function takes values in $a_h\varphi(t)\G$. We recall that $a_h\varphi(t)\ne0$, so that this set consists of invertible matrices.
	Using periodicity of $v$ and the boundary conditions on $\gamma$ we obtain
	\begin{equation*}
		w(0)=v(0)(\gamma(0))^{-1}=v(0), \hskip1cm
		w(2\pi)=v(2\pi)(\gamma(2\pi))^{-1}=v(0)G^{-1}.
	\end{equation*}
	However, as $w$ is continuous on $[0,2\pi]$, it is constant.
	Since $v(0)$ is an invertible matrix, this gives $G=\Id$, a contradiction.
\end{proof}

\subsection{Functional setting II: Metric-space valued $\BV$ functions}
We recall from \cite{Amb90SBVX} the definition of functions of Bounded Variation into a general separable locally compact metric space $(X,\dX)$. 
\begin{definition}
	\label{def-BV-metric}
	Let 
	$u\in L^1(\Omega ; X)$. We say that $u$ belongs to $\BV(\Omega; X)$ if for every $\varphi \in \Lip(X;\R^m)$ one has $\varphi\circ u \in \BV(\Omega;\R^m	)$ and there exists a finite Borel measure $\sigma$
(which may depend on $m$, but not on $\varphi$)
	such that
	\begin{equation}\label{f:BVX}
		|D(\varphi\circ u)| \leq \Lip(\varphi) \,\sigma.
	\end{equation}
In this case,
for $m\in\N_{>0}$
the total variation measure of $u$ is defined as
	\begin{equation}
		\label{eq-sup-measures}
		|D u|_{(m)}:=\sup\{|D(\varphi\circ u)|\,:\ \varphi\in \Lip_1(X;\R^m)\},
	\end{equation}
	where the supremum of measures is taken as in Definition 1.68 in \cite{AmbFusPal00}.
\end{definition}
As for the case of Sobolev functions, the value of $|Du|_{(m)}$ depends on $m$ (see Remark~\ref{rem-dep-m}).

The usual decomposition of $|Du|$ holds true in $\BV(\Omega;X)$ (see Theorems 2.2 and 2.3, and Remark 2.3 in \cite{Amb90SBVX}):
\begin{equation}
	\label{eq-gradu-decomposition}
	|Du|_{(m)}= |\nabla u|_{(m)} \,\calL^N + \dix(u^+,u^-) \, \calH^{N-1}\LL J_u + |Du|_{(m)}^c,
\end{equation}
where $ |\nabla u|_{(m)} (x)=\sup\{|\nabla(\varphi\circ u)|(x) \colon \varphi\in \Lip_1(X;\R^m) \}\in L^1(\Omega)$,
$J_u$ is a countably $\calH^{N-1}$--rectifiable subset of $\R^N$ with unit normal $\nu(x)$ defined for $\calH^{N-1}$-a.e.\ $x \in J_u$, and the traces $u^\pm(x)=u^\pm(x, \nu(x))$ are characterized by
\begin{equation}\label{f:upm}
\lim_{\rho \to 0^+} \frac{B_\rho(x) \cap \{y\in\Omega \colon\ \pm(y-x)\cdot \nu(x)>0,\ \dix(u(y),u^\pm(x))>\eps\}}{\rho^d}=0 \qquad \forall \eps>0.
\end{equation}

\begin{remark}
	\label{rem-BV-SBV}
	Following \cite{Amb90SBVX}, $u\in \SBV(\Omega; X)$ if $u\in \BV(\Omega; X)$ and $|Du|^c_{(m)}=0$. This is equivalent to require that $u\in \BV(\Omega; X)$ and  $\varphi\circ u\in \SBV(\Omega; \R^m)$ for every $\varphi\in \Lip(X;\R^m)$ (see
\cite[Remark 2.4]{Amb90SBVX}). 
		\end{remark}
	
As in the case of $X$-valued Sobolev functions, the norm depends on $m$, but the space does not. In working with $\Mno/\G$ we shall mostly take $m=d^2$ and, when there is no risk of confusion, simply write $|Du|$ for $|Du|_{(d^2)}$.

	\begin{remark}
		\label{rem-SVB0-Sob}
	The decomposition \eqref{eq-gradu-decomposition} and definition \eqref{eq-sup-measures} also  allow to characterize subspaces of $\BV(\Omega;X)$.
	
	To this aim we recall the following basic property of the derivation of measures, also used in \cite{Amb90SBVX}:  if $\sigma$ is a positive measure, and $\mu$, with $\mu(\Omega)<+\infty$, is the supremum of a countable family of positive measures $\mu_h$, and $\mu_h=f_h\sigma +\mu_h^s$,  where $\mu_h^s$ denotes the singular part of $\mu_h$ with respect to $\sigma$ and $f_h\in L^1(\Omega;\sigma)$  (i.e. $f_h$ is in $L^1$ with respect to the measure $\sigma$),  then
	\begin{equation}
		\label{eq-decompos-measure}
		\mu=f\sigma +\mu^s\,,\qquad  \hbox{with } f=\sup f_h\qquad \sigma\EEE\hbox{-a.e.\ in } \Omega. 
	\end{equation}
	Therefore from  Remark~\ref{rem-BV-SBV}  we have
	\begin{equation}
		\label{eq-lip-SVB}
		u\in \SBV(\Omega;X)\qquad \Longleftrightarrow \qquad  \varphi\circ u\in \SBV(\Omega;\R^m)\ \hbox{for all } \varphi\in \Lip(X;\R^m).
	\end{equation}
	Moreover we also define the space  $\SBV_0(\Omega;X)$ as the set of functions $u\in \SBV(\Omega;X)$ such that $|\nabla u|=0$, which is also equivalent to
	\begin{equation}
		\label{eq-lip-SVB0}
	\varphi 	\circ u\in \SBV_0(\Omega;\R^m)\ \hbox{for all } \varphi\in \Lip(X;\R^m).
	\end{equation}
	Finally, when $u \in \SBV(\Omega; X)$ has no jump, and  $|\nabla u|_{(m)} \in L^p(\Omega)$, the notion above complies with the definition of $W^{1,p}(\Omega; X)$ in \cite{Resh} recalled in Definition~\ref{def-sobolev-metric}.
\end{remark}

We recall here the compactness results in $\BV$  proven in \cite[Theorem 2.4]{Amb90SBVX}.

\begin{theorem}\label{t:BVXcomp}
	Let $(u_n)_{n\in\N}$ be a sequence of functions in $\BV(\Omega;X)$, and $m\in\N_{>0}$.
	If there exists $z_0\in X$ such that
	\[
	\sup_{n\in\N} \left( |Du_n|_{(m)}(\Omega)+\int_\Omega \dix(u_n(x),z_0)\, \dx \right)<+\infty,
	\]
	then there exist a subsequence $(u_{n_k})_k$ and a function $u\in \BV(\Omega;X)$ such that
	$u_{n_k}$ converges to $u$ a.e.\ in $\Omega$, and
	\[
	|Du|_{(m)}(\Omega) \leq \liminf_{k \to+\infty}  |Du_{n_k}|_{(m)}(\Omega).
	\]
\end{theorem}

We also recall that one can characterize functions in $\BV(\Omega;X)$ via slicing as follows.
We fix $\xi\in \SN:= \{ \xi\in\R^N\colon\ |\xi|=1\}$ and we set
\[
\Pi^\xi:=\{y\in\R^N\colon\ y\cdot \xi=0\}, \quad 
\Omega_y^\xi:=\{t\in\R\colon\ y+t\xi \in\Omega\}, \quad
\Omega^\xi:=\{y\in\Pi^\xi \colon \Omega_y^\xi \neq \emptyset\}.
\]
For any function $u\colon \Omega \to X$ we define  the slices $u_y^\xi\colon \Omega_y^\xi\to X$ by
$$
u_y^\xi(t):=u(y+t\xi).
$$
 The following result is an immediate consequence of \cite[Proposition 2.1]{Amb90SBVX}, stated with our notation and also formulated for Sobolev functions.

\begin{proposition}
	\label{prop-slicing}
The following hold:
	\begin{itemize}
		\item[i)] if $u\in \BV(\Omega;X)$, then for every $\xi\in \SN$ we have
\begin{equation}\label{3004251203}
		u_y^\xi\in \BV(\Omega_y^\xi; X)\quad \hbox{and} \quad 	|D u|_{(1)}\geq |D(u\xy)| \quad  \text{as measures in }\Omega\xy, \EEE \text{ for }\hN\text{-a.e.\ }y\in \Omega^\xi;
	\end{equation}
\item[ii)]	if $u\in W^{1,p}(\Omega; X)$, then for every $\xi\in \SN$ we have $u_y^\xi\in W^{1,p}(\Omega_y^\xi; X)$ and
$$
|\nabla u (x)|\geq |(u\xy)'(t_x)| \quad\text{for }\hN\text{-a.e.\ }y\in \Omega^\xi\quad \hbox{and} \quad \text{a.e.\ }x\in \Omega,
$$
where $t_x \in \Omega\xy$ is such that $x=y+t_x \xi$.
		\end{itemize}
\end{proposition}

\begin{proof}
	The proof of i) follows from \cite[Proposition 2.1]{Amb90SBVX} and a localization argument.
	Property ii) follows from i) and the characterization of $|Du|$ given in \eqref{eq-gradu-decomposition}.
\end{proof}

We prove
that maps in $\SBV(\Omega;\OnG)$ have a lifting in $\SBV(\Omega;\On)$, with a bound on the norm. This
could be obtained easily as a consequence of
 a general result by
Canevari and Orlandi \cite{CanOrl20JFA}
on maps valued in compact manifolds with a universal covering.
In our case, the manifold is $\OnG$, and permits a simpler argument.
We remark that in the case $d=2$, this also follows from the simple proof in \cite{DavilaIgnat2003lifting} for maps valued in $\mathbb{S}^1$. In higher dimension, however, a localization is needed, as in \cite{CanOrl20JFA}.
Our proof applies without significant changes both for $\SOn$ and for $\On$;
we present it only in the case of $\On$, which is the one used below.

\begin{theorem}\label{theoliftingmanifold}
 Let $u\in \SBV(\Omega;\OnG)$, $\Omega\subset\R^N$ open and bounded. Then there is $\beta\in \SBV(\Omega;\On)$ such that
 \begin{equation}
 u=[\beta] \text{ pointwise a.e. in $\Omega$,  and } |D\beta|(\Omega)\le C |Du|_{(d^2)}(\Omega).
  \end{equation}
The constant depends only on $d$ and $\G$.
\end{theorem}
\begin{proof}
 Let $r_0:=\frac14 \min\{|G-G'|: G, G'\in\G, G\ne G'\}$.

 \emph{Step 1. We can identify $B_{r_0}([Z])$ with
 $B_{r_0}(Z)$.}
Fix $Z\in \OnG$. We check that $Y\mapsto [Y]$ is an isometric bijection of
$B_{r_0}(Z)\subset\On$ onto $B_{r_0}([Z])\subset\OnG$.
To see this, the key fact is that if $Y$, $Y'\in B_{r_0}(Z)$ and
$G\in \G\setminus\{\Id\}$ then
\begin{equation}\label{eqgyyp}
 |GY-Y'|> |Y-Y'|,
\end{equation}
which implies that $\dG([Y],[Y'])=|Y-Y'|$.
In turn, to show \eqref{eqgyyp} it suffices to observe that, since $Y\in\On$
and $|Y-Y'|\le |Y-Z|+|Y'-Z|<2r_0$,
\begin{equation}
 |GY-Y'|=|(G-\Id)Y+(Y-Y')|\ge |G-\Id|-|Y-Y'|
\ge 4r_0-2r_0>|Y-Y'|.
\end{equation}
We denote by  $\psi_Z:B_{r_0}([Z])\to B_{r_0}(Z)\subset\On$ the inverse of $Y\mapsto [Y]$.

\emph{Step 2. Localization via the coarea formula.}
  We claim that for any $Z\in \On$ there is $r\in (r_0/2,r_0)$ such that
 \begin{equation}\label{eqchoicer}
  \calH^{N-1}(\partial^*(u^{-1}(B_r([Z]))))<  c |Du|(\Omega).
 \end{equation}
To see this, consider $\varphi:\OnG\to\R$ defined by
$\varphi(Y):=\dG(Y,[Z])$, and
notice that $u^{-1}(B_t([Z]))=\{x: \varphi(u(x))<t\}$.
The function $\varphi$ is obviously $1$-Lipschitz, therefore $\varphi\circ u \in \SBV (\Omega)$, and 
$|D(\varphi\circ u)|(\Omega)\le |Du|(\Omega)$. By the coarea formula applied to the $\SBV$  function $\varphi\circ u: \Omega\to\R$ we obtain
\begin{equation}
 \int_\R  \calH^{N-1}(\partial^*(\{\varphi\circ u<t\})) \dt
 = |D(\varphi\circ u)|(\Omega),
\end{equation}
therefore we can choose $r$ as in \eqref{eqchoicer}, with $c=2/r_0$ depending only on $\G$.

\emph{Step 3. Decomposition of the domain.}
Let $Z_1, \dots, Z_M\in \On$ be such that $\OnG\subset \bigcup_i B_{r_0/2}([Z_i])$
(this choice depends only on $\G$), and let $r_i$ be the corresponding values with the property \eqref{eqchoicer} (which depend also on $u$).
We define iteratively 
\[
\omega_1:= u^{-1}(B_{r_1}([Z_1])),\
\text{and}\ \omega_i:= u^{-1}(B_{r_i}([Z_i]))\setminus \bigcup_{j<i} \omega_j.
\]
The sets are a partition of $\Omega$,
and $\calH^{N-1}(\partial^*\omega_i)\le c |Du|(\Omega)$ for each $i$, with $c$ depending only on the group $\G$.

\emph{Step 4. Conclusion.}
We set $\beta:=\psi_{Z_i}\circ u$ in each $\omega_i$. Obviously $[\beta]=u$,
and inside each $\omega_i$ we have $|D\beta|\le C|Du|_{(d^2)}$ as measures.
We compute
\begin{equation}
 |D\beta|(\Omega)\le \sum_{i=1}^M |D\beta|(\omega_i)+ \sum_{i=1}^M\|\beta\|_{L^\infty} \calH^{N-1}(\partial^*\omega_i) \le C |Du|_{(d^2)}(\Omega),
\end{equation}
again with $C$ depending only on $\G$ and $d$.
\end{proof}

\begin{remark}\label{rem:0904251748}
	By a localization argument, it is simple to see that if $u\in \PC(\Omega;\On/\G)$, then the function $\beta$ whose existence is guaranteed by
	Theorem~\ref{theoliftingmanifold} satisfies $\beta \in \PC(\Omega; \On)$. Moreover, if any other $\beta_*$ with this property exists, then $\beta_*=G\beta$ for some $G\in \BV(\Omega;\G)$.
\end{remark}

\begin{remark}\label{ref:lifting}
An extension of the argument used in Theorem~\ref{theoliftingmanifold} to construct a lifting of a map in
 $\SBV(\Omega;\OnG)$  to a map in   $\SBVG(\Omega;\On)$ permits to show
 that if $\G\subset\SO(2)$ then any map in
 $\SBV(\Omega;\Mdue/\G)$ admits a representative in  $\SBV(\Omega;\Mdue)$.
 The treatment of the higher-dimensional case is more complex and outside of the scope of this paper. Therefore we postpone a more detailed analysis of the relation between the spaces
 $\SBV(\Omega;\Mno/\G)$ and $\SBV(\Omega;\Mno)$ to future work.
\end{remark}

\subsection{Metric-space valued phase field models}
 With the above functional framework  at hand,  we can formulate the approximating phase field model  in terms of order parameters valued into equivalence classes. The natural domain will be given  by Sobolev functions with values in $\Mno/\G$. We define
\begin{equation}\label{eq-tilde-F}
	\widetilde \F_\varepsilon(u,v;A):= 
	\int_{{A} } f_\varepsilon^2(v) \,|\nabla u|_{(d^2)}^2 \dx + \AT_\e(v;A) +\frac{1}{\delta_\varepsilon} \int_A \dG^2(u, \On/\G) \dx
\end{equation}
for  $A\subset\Omega$ open and
\[
(u,v)\in \ H^1(\Omega;\Mno/\G) \times  H^1(\Omega;[0,1]),
\] 
and equal to $+\infty$ otherwise in $L^1(\Omega;\Mno/\G)\times  L^1(\Omega)$.

By Remark \ref{r:sbvgvsh1}, any $\beta \in \SBVG(\Omega; \Mno)$ is such that $[\beta]\in H^1(\Omega;\Mno/\G)$ with $|\nabla[\beta]|_{(d^2)} \leq |\nabla \beta|$, and obviously $\dG([\beta], \On/\G) = \dist(\beta, \On)$. It follows that
\begin{equation}\label{0904251306}
	\tF_\varepsilon([\beta],v;A)\leq \F_\varepsilon(\beta,v;A).
\end{equation}
Finally in view of the invariance \eqref{eq-limit-classi} and Remark \ref{rem:0904251748} we can define 
\begin{equation}\label{1304251245}
	\tF^\lambda(u,1):=\F^\lambda(\beta,1) \quad\text{for every } u \in \PC(\Omega; {\On}/\G)
\end{equation}
and $\tF^\lambda(u,v):=+\infty$ otherwise in $L^1(\Omega; {\Mno}/\G)\times L^1(\Omega)$ where, for any $u \in \PC(\Omega; \OnG)$, $\beta \in \PC(\Omega; \On)$ is a lifting of $u$ as in Theorem~\ref{theoliftingmanifold}. The invariance \eqref{eq-limit-classi} guarantees that the result does not depend on the choice of the lifting.

In this general framework we can now state the counterpart of Theorem \ref{thm:main}.

\begin{theorem}\label{thm:maineq}
	The following properties hold:
	\begin{itemize}
		\item[(C)'] For any $(u_\varepsilon, v_\varepsilon)_\varepsilon \subset L^1(\Omega;\Mno/\G \times [0,1])$ satisfying
		\begin{equation}\label{compesteq}
			\limsup_{\varepsilon\to 0} \tF_\varepsilon(u_\varepsilon, v_\varepsilon)<+\infty
		\end{equation}
		there exists $u \in \PC(\Omega; \On/\G)$ such that, up to a subsequence,
		\begin{equation}\label{compeq}
			u_\varepsilon \to u \text{ in }L^1(\Omega; \Mno/\G),\qquad v_\varepsilon \to 1 \text{ in }L^1(\Omega).
		\end{equation}
\item[(LB)'] For any $(u_\varepsilon, v_\varepsilon)_\varepsilon
\subset L^1(\Omega;\Mno/\G\times [0,1])$ and $u \in \PC(\Omega; \On/\G)$ satisfying \eqref{compeq}, it holds
		\begin{equation*}\label{liminfeq}
			\tF^\lambda(u,1)\leq \liminf_{\varepsilon \to 0} \tF_\varepsilon(u_\varepsilon, v_\varepsilon).
		\end{equation*}
\item[(UB)'] For any $u \in \PC(\Omega; \On/\G)$ there exists $(u_\varepsilon, v_\varepsilon)_\varepsilon \subset L^1(\Omega;\Mno/\G) \times L^1(\Omega;[0,1])$ satisfying \eqref{compeq} and
		\begin{equation}\label{limsupeq}
			\tF^\lambda(u,1)\geq \limsup_{\varepsilon \to 0} \tF_\varepsilon(u_\varepsilon, v_\varepsilon).
		\end{equation}
	\end{itemize}
\end{theorem}

The proof of this result is obtained as a consequence of the analysis performed in the rest of the paper. Compactness (C') is stated and proved in Theorem \ref{t:comp1}. The lower bound (LB') is stated and proved in Theorem \ref{th-lb}. The upper bound is a consequence of the corresponding upper bound in Theorem \ref{thm:main}, proved in Section \ref{sec:proofmain}, Theorem \ref{th-upperbound}, and the fact that $\tF_\e\leq \F_\e$ as in \eqref{0904251306}.

\section{Compactness}\label{subsec:comp}

We start with a truncation argument, in order to restrict our attention to sequences equibounded in $L^\infty$. We denote by $\Phi(\xi)$ the orthogonal projection of a matrix $\xi\in\Mno$ on the closed ball  $\overline B_{\sqrt d}:=\{\eta\in\Mno\colon |\eta|\leq  \sqrt d\}$,
\begin{equation}\label{f:proj}
	\Phi(\xi):=\begin{cases}
		\xi\quad &\hbox{if } |\xi|\leq  \sqrt d,\\
		\displaystyle{\sqrt d \frac{\xi}{|\xi|}} \quad &\hbox{if } |\xi|> \sqrt d.
	\end{cases}
\end{equation}
Since $\Phi(G\xi)=G \Phi(\xi)$ for all $G \in \On$,
$\Phi$ can be identified with a map from $\Mno/\G$ to itself.

\begin{lemma}\label{l:linfty}
	Let $(u_\varepsilon)_\varepsilon \subset H^1(\Omega; \Mno/\G)$ be such that
	\begin{equation}\label{1104251323}
	\lim_{\e\to0}	\int_{{\Omega} } \dG^2(u_\varepsilon, \On/\G) \dx =0.
	\end{equation}
Then, for every 
	$(v_\varepsilon)_\varepsilon\subset L^1(\Omega; [0,1])$ we have that 
	\begin{equation}\label{F:trunk}
		\tF_{\e}({\Phi}(u_\e),v_\e)\leq \tF_{\e}(u_\e,v_\e), \qquad \lim_{\e \to 0} 	\|{\Phi}(u_\e)-u_\e\|_{L^1(\Omega; \Mno/\G)}=0.
	\end{equation}
A corresponding assertion holds for 
$\F_\e$ and 
any $(\beta_\eps)\subset \SBVG(\Omega;\Mno)$.
\end{lemma}

\begin{proof}
It follows immediately from the definition of $\Phi$, and the fact that
		$\dist(\Phi(\xi), \On)\leq \dist(\xi, \On)$ for any matrix $\xi$.
	 Moreover, $\Phi$ is 1--Lipschitz, and hence, by  Definition~\ref{def-sobolev-metric} we deduce that
	$|\nabla {\Phi}(u_\e)| \leq |\nabla u_\e|$.
	Finally, 
	since $\On\subset \overline B_{\sqrt d}$ and $\Phi$ is a projection,
	\[
	|{\Phi}(\xi)-\xi|
	\le \dist(\xi, \overline B_{\sqrt d})
	\le \dist(\xi, \On).\qedhere
	\]
\end{proof}
We first prove compactness  in Theorem~\ref{thm:maineq} \EEE by using the theory of $\BV$ functions valued in metric spaces and the strategy from \cite{ConFocIur25}.

\begin{theorem}\label{t:comp1}
	For any $(u_\varepsilon, v_\varepsilon)_\varepsilon \subset L^1(\Omega;\Mno/\G \times [0,1])$ satisfying \eqref{compesteq}
	there exists $u \in \SBV_0(\Omega; \On/\G)$ such that, up to a subsequence,
	\begin{equation*}
		u_\varepsilon \to u \text{ in }L^1(\Omega; \Mno/\G),\qquad v_\varepsilon \to 1 \text{ in }L^1(\Omega).
	\end{equation*}
	In particular, (C)' in Theorem~\ref{thm:maineq} is true.
\end{theorem}

\begin{proof}
Let us fix a vanishing sequence $(\varepsilon_n)_n$ and denote  $\tF_n\equiv\tF_{\varepsilon_n}$ \EEE for every $n \in \N$.
	Let $(u_n, v_n)_n \subset L^1(\Omega;  \Mno/\G \EEE \times [0,1])$ be such that
\begin{equation}\label{f:enest}
    \sup_n \tF_n(u_n, v_n) \leq  K.
\end{equation}
	The convergence of $v_n$ to $1$  directly \EEE follows from the estimate 
\begin{equation}\label{eq1mvnepsto0}
\int_\Omega\frac{(1-v_n)^2}{4\eps_n}\dx\leq K.
\end{equation}
	Moreover,   $u_n \in H^1(\Omega; \Mno/\G)$ \EEE by \eqref{f:enest}, and,
	by Lemma \ref{l:linfty}, we may assume
	that 
	\begin{equation}\label{f:lixc}
\|u_n\|_{L^\infty}\le\sqrt d \qquad \text{for every } n\in\N.
	\end{equation}
We follow the strategy from \cite[Prop.~4.11]{ConFocIur25}.
As a first step, we need to identify small subsets of $\Omega$, with vanishing measure and bounded perimeter, where 
the the energy bound cannot provide any uniform bound of $|\nabla u_n|$. More precisely, fixed $\rho\in (0,1/2)$, for every $n$ we can choose $t_n^\rho\in (1-\rho, 1-\rho/2)$ such that
\begin{equation}\label{eqcptpropenrho}
	E_n^\rho:= \{v_n<\lambda_n^\rho\}\hskip2mm
	\text{ obeys }\hskip2mm
	\Per(E_n^\rho, \Omega)\le  C_\rho K
	\hskip2mm\text{ and }\hskip2mm
	\LN(E_n^\rho)\le C_\rho \eps_n K.
\end{equation}
Namely, for
$\phi(t):=(1-t)^2/2$, by the coarea formula, Young's inequality and \eqref{f:enest}, we get
\begin{equation}\label{eqestphit}
\int_\R \Per(\{\phi(v_n)>t\}, \Omega) \dt
=\int_\Omega |\nabla(\phi(v_n))| \dx \leq  \int_\Omega\left(\frac{(1-v_n)^2}{4\eps_n}+\eps_n|\nabla v_n|^2 \right) \dx \leq K,
\end{equation}
and it is enough to choose $\lambda^\rho_n=\phi^{-1}(t^\rho_n)$, with $t^\rho_n \in (\phi(1-\rho/2), \phi(1-\rho))$ such that
\[
\left(\phi(1-\rho)-\phi(1-\rho/2)\right) \Per(\{\phi(v_n)>t^\rho_n\}, \Omega) \leq  \int_{\phi(1-\rho/2)}^{\phi(1-\rho) } \Per(\{\phi(v_n)>t\}, \Omega) \dt.
\]
The second estimate in \eqref{eqcptpropenrho} then follows from the fact that
\[
\LN(\{v_n<\lambda_n\})\frac{(1-\lambda_n)^2}{4\varepsilon_n} \leq \int_{\{v_n<\lambda_n\}} \frac{(1-v_n)^2}{4\eps_n}  \dx
\leq K. 
\]
Now we use again \eqref{f:enest}  to get a uniform estimate for $u_n$ in $\Omega \setminus E_n^\rho$,  exploiting the fact that $v_n \geq 1-\rho$ in this set. Namely, setting $B_n:=\{ f(v_n)>M_{\varepsilon_n} \varepsilon_n^{-\frac12} \}$, we have
\[\begin{split}
K\geq &  \tF_n(u_n, v_n; \Omega\setminus E^\rho_n) \geq M_{\varepsilon_n}^2 \int_{B_n\setminus E_n^\rho}    |\nabla u_n|^2 \dx  \\  + &
\int_{\Omega \setminus B_n\setminus E_n^\rho }  \left( \e_nf^2(v_n) |\nabla u_n|^2 \dx  +
\frac{(1-v_n)^2}{4\e_n}  \right)\dx \\ &
\geq M_{\varepsilon_n}^2 \int_{B_n\setminus E_n^\rho}  |\nabla u_n|^2 \dx +
\int_{\Omega \setminus B_n\setminus E_n^\rho }   f(v_n) (1-v_n) |\nabla u_n| \dx, 
\end{split} \] 
so that 
\begin{equation}\label{f:stEnc}
\int_{{\Omega \setminus E_n^\rho } }
|\nabla u_n| \dx\leq \frac{(K\calL^N(\Omega))^{\frac 12}}{M_{\varepsilon_n}}+\frac{K}{m_\rho},
\end{equation}
where $0<m_\rho:=\min_{s\in [1-\rho,1)}(1-s)f(s)$
obeys $\lim_{\rho\to0} m_\rho=\infty$. 
 In conclusion, setting
	\begin{equation}\label{01203242202}
		\ol u_n^\rho:=u_n \chi_{\Omega\sm E_n^\rho} +[\Id]\,  \chi_{E_n^\rho} ,
	\end{equation}
	we have $\ol u_n^\rho \in \SBV(\Omega;\Mno/\G)$ with
	\begin{equation}\label{02401241158}
		J_{\ol u_n^\rho}\subset  \partial^* E_n^\rho, \qquad  |\nabla \ol u_n^\rho| = |\nabla u_n| \chi_{\Omega\sm E_n^\rho}= |\nabla (u_n \chi_{\Omega\sm E_n^\rho})|.
	\end{equation}
By \eqref{f:lixc},  \eqref{eqcptpropenrho},
\eqref{02401241158},
\eqref{f:stEnc} and Theorem~\ref{t:BVXcomp} we
obtain that there exist a function  $u^\rho \in \BV(\Omega;\Mno/\G)$ and a subsequence $(\ol u_{n_k}^\rho)_k$ converging to $u^\rho$ in $L^1(\Omega;  \Mno  /\G)$.
Since \eqref{f:enest} implies \eqref{1104251323}, we obtain that $u^\rho$ takes values in $\On/\G$ almost everywhere.

Moreover, since $\LN(E_n^\rho)\le C_\rho \eps_n K\to0$ and $\|u_n\|_{L^\infty}\le \sqrt d$,
we have $\|u_n-\ol u^\rho_n\|_{L^1}\to0$, therefore also the subsequence
$(u_{n_k})_k$ converges to the same $u^\rho$. In particular,
taking a diagonal subsequence we can ensure that the limit does not depend on $\rho$, we denote it by $u$.
	
Let $\varphi\in \Lip_1(\Mno/\G)$. Then for any $\rho$ and $n$ we have
$\varphi\circ \ol u^\rho_n\in BV(\Omega)$;
from \eqref{eqcptpropenrho} and
\eqref{02401241158} we obtain
$\sup_n \calH^{N-1}(J_{\varphi\circ \ol u^\rho_n})<\infty$. Therefore, recalling \eqref{f:stEnc}, and exploiting a standard localization argument, e.g. arguing by slicing to reduce to a 1-dimensional domain where the approximating functions have the same (finite) number of jump points, we get
\begin{equation}\label{eqcptsbv0}
	\begin{split}
|D^c(\varphi\circ u)|(\Omega)
+ \|\nabla (\varphi\circ u)\|_{L^1(\Omega)} &
\le \liminf_{k\to\infty} \|\nabla (\varphi\circ \ol u_{n_k}^\rho)\|_{L^1(\Omega)} \\ &
\le \liminf_{k\to\infty} \|\nabla u_{n_k}\|_{L^1(\Omega\setminus E_n^\rho)}
\le \frac{K}{m_\rho}.
\end{split}
\end{equation}
We recall that $\lim_{\rho\to0} m_\rho=\infty$.
Since the left-hand side does not depend on $\rho$, it is zero, and we obtain $\varphi\circ u\in \SBV_0$. By
Remark~\ref{rem-SVB0-Sob}, $u\in \SBV_0(\Omega;\On/\G)$.
\end{proof}

We remark that the compactness in Theorem~\ref{thm:main} is a direct consequence of Theorem~\ref{t:comp1}, owing to the properties in Remark~\ref{r:sbvgvsh1} and
Theorem~\ref{theoliftingmanifold} on the lifting of $\BV$ functions with values in $\On/\G$, applied to any limiting field $u$.

Nevertheless, we provide a different proof of Theorem~\ref{thm:main} which does not make use of the theory of $\BV$ functions valued in metric spaces.
We exploit the fact that we may apply
Theorem~\ref{theoliftingmanifold} to functions valued in $\On/\G$, target set of the limits.
The crucial step is to modify the original sequence in such a way that the modifications are valued into $\On/\G$ and
still have uniformly bounded energy.
By this approach we prove weak$^*$ convergence in $\BV$, up to a small perturbation.

\begin{theorem}\label{t:comp2}
	Let $(\beta_\varepsilon, v_\varepsilon)_\varepsilon \subset L^1(\Omega;\Mno \times[0,1])$ satisfy
	\eqref{compest}. Then
	there exists $\beta \in \SBV_0(\Omega; \On)$, a subsequence $({\e_n})_n$ and  $G_n\in L^\infty(\Omega;\G)$ such that
	$G_n\beta_{\e_n}$  belong to $\SBVG(\Omega; \Mno)$ and converge to $\beta$ in $L^1(\Omega;\Mno)$,
	and $v_{\e_n}\to1$ in $L^2(\Omega)$.
	In particular, (C) in Theorem~\ref{thm:main} is true.
\end{theorem}
\begin{proof}
	By Lemma \ref{l:linfty} we may assume that  
	\begin{equation}\label{2904251925}
		|\beta_\varepsilon|\leq \sqrt{d} \quad \text{a.e.\ in }\Omega.
	\end{equation}
	Up to passing to a subsequence $(\beta_{\varepsilon_n})_n\equiv (\beta_n)_n$, by \eqref{compest} we get
	\begin{equation}\label{0305251910}
		\sup_n \F_{\eps_n}(\beta_n, v_n)\leq K.
	\end{equation}
	We deal with $(v_n)_n$ as in the proof of Theorem~\ref{t:comp1}; in particular, $v_n\to1$ a.e.\ in $\Omega$ and 
	for any $\rho\in (0,1/2)$
	there are suitable $t_n^\rho \in (1-\rho,1-\rho/2)$ such that $E_n^\rho:=\{v_n <t_n^\rho\}$ satisfies for every $n\in \N$ the properties \eqref{eqcptpropenrho}, and 
the same computation as for \eqref{f:stEnc}
leads to
	\begin{equation}\label{02401241218'}
	\int_{\Omega\sm E_n^\rho}|\nabla \beta_n| \dx 
	\leq \frac{(K\calL^N(\Omega))^{\frac 12}}{M_{\varepsilon_n}}+\frac{K}{m_\rho}.
\end{equation}
The next part of the argument differs from the one in Theorem~\ref{t:comp1}. We
consider the functions 
\begin{equation}\label{2401241151}
    h_n(x):= \beta_n^T(x) \beta_n(x), \qquad n\in\N.
\end{equation}
We have that $h_n \in \SBV(\Omega;\Mns)$
with $\|h_n\|_{L^\infty}\le d$, $\Mns$ denoting the set of $d\times d$ symmetric matrices, and $J_{h_n} \subset J_{\beta_n}$. Moreover,
if $x\in J_{h_n}$, and $G\in\G$ is such that $\beta^+_n(x)=G\beta^-_n(x)$, then 
\[
h_n^+(x)= \left(\beta_n^+(x)\right)^T\beta_n^+(x)= \left(G\beta_n^-(x)\right)^TG\beta_n^-(x)=  \left(\beta_n^-(x)\right)^T\beta_n^-(x)
=h_n^-(x),
\]
therefore $h_n\in W^{1,1}(\Omega;\Mns)$,
with $|\nabla h_n|\le C |\nabla \beta_n|$ a.e.\ in $\Omega$.
	
	A similar computation shows that 
	\[
	w_n:=|\det \beta_n| \in W^{1,1}(\Omega),
	\]
	with $|\nabla w_n|\le C |\nabla\beta_n|$ a.e.\ in $\Omega$.
	By the Coarea Formula, for every $n\in \N$ there exists $\lambda'_n\in (\frac14,\frac12)$ such that
	\begin{equation}\label{2401241247}
		\Per(\{ |\det \beta_n| < \lambda'_n\}, \Omega ) \leq C.
	\end{equation}
	Moreover, since $ \lambda'_n \le 1/2$, and
	$\min\{\dist(A,\On): |\det A|\le1/2\}>0$,
%
%
	\[
	\LN\big( \{ |\det \beta_n| < \lambda'_n\}  \big)
\leq C \int_\Omega \dist^2(\beta_n, \On) \dx \leq CK \delta_{\varepsilon_n}.
	\]
	Hence, the sets $F_n^\rho:= E_n^\rho \cup \{ |\det \beta_n| < \lambda'_n\}$
obey
\begin{equation}\label{eqcptpropenrhoF}
	\Per(F_n^\rho, \Omega)\le  C+C_\rho K
	\hskip2mm\text{ and }\hskip2mm
\LN(F_n^\rho)\le C_\rho \eps_n K+
CK \delta_{\varepsilon_n}
\end{equation}
and the functions
	\begin{equation}\label{2401241351}
		\ol h_n:= h_n \chi_{\Omega \sm F_n^\rho} + \Id\,  \chi_{F_n^\rho}
	\end{equation}
	satisfy $\ol{h}_n \in \SBV(\Omega; \Mno_\sym)$  with
	\begin{equation}\label{2401241353}
		|D\ol{h}_n|(\Omega)+ \|\ol{h}_n\|_{L^\infty(\Omega; \Mno)}\leq C, \qquad \det(\ol{h}_n)\geq \frac{1}{16} \text{ a.e.\ in}\ \Omega.
	\end{equation}
The maps $A\mapsto A^2$ and $A\mapsto A^{-1}$ are smooth bijections of $\Mnsp$,
i.e., the set of strictly positive-definite matrices in $\Mns$, onto itself  with smooth inverse, and in particular their inverses $A\mapsto A^{1/2}$ and $A\mapsto A^{-1/2}$ are smooth, and then Lipschitz on compact sets.
Therefore  the functions $(\ol{h}_n)^{-\frac{1}{2}}$, defined by $(\ol{h}_n)^{-\frac{1}{2}}(x):=\left((\ol{h}_n(x)^{-1})\right)^{\frac{1}{2}}$, are  in $\SBV(\Omega; \Mns)$,  and
	\begin{equation}\label{2801241031}
		J_{(\ol{h}_n)^{-\frac{1}{2}}}=J_{\ol{h}_n}\subset \partial^* F_n^\rho, \qquad
		|\nabla \big((\ol h_n)^{\frac{1}{2}}\big)| +|\nabla \big((\ol h_n)^{-\frac{1}{2}}\big)| \leq C |\nabla \ol h_n| \quad\text{a.e.\ in }\Omega\sm F_n^\rho.
	\end{equation}
	In particular, the functions $k_n\colon \Omega \to \On$ defined by
	\begin{equation}\label{2602241346}
		k_n:=\beta_n(h_n)^{-\frac{1}{2}}\chi_{\Omega\sm F_n^\rho}+ \Id \chi_{F_n^\rho},
	\end{equation} 
by \eqref{2401241353}, \eqref{2801241031} satisfy
	\begin{equation*}\label{2801241258}
		|\nabla  k_n|\leq C |\nabla  \beta_n| \quad\text{a.e.\ in }\Omega.
	\end{equation*}
 From  \eqref{02401241218'}  and since $k_n=\mathrm{Id}$ in $F_n^\rho$ (so in particular in $E_n^\rho\subset F_n^\rho$),
we deduce
that $k_n\in \SBV(\Omega; \On)$ with
	\begin{equation}\label{2801241247}
\int_\Omega	|\nabla k_n | \dx \leq
\frac{C K^{\frac 12}}{M_{\varepsilon_n}}+\frac{C K}{m_\rho},
\qquad J_{ k_n}\subset J_{\beta_n}\cup \partial^* F_n^\rho, \qquad
k_n^+\in \G  k_n^- \text{ on }J_{\beta_n}
\setminus F_n^\rho.
	\end{equation}
	Passing to the equivalence classes on the image, the functions defined by $[k_n](x):=[k_n(x)]$ for $x\in\Omega$ belong to $\SBV(\Omega; \On/\G)$, and
	\[
|D([k_n])|_{(d^2)}(\Omega)\leq \int_\Omega	|\nabla k_n | \dx+ \|k_n\|_\infty \Per(F_n^\rho)\leq C.
	\]
By Theorem~\ref{theoliftingmanifold}, for every $n\in\N$ we find $G_n\in L^\infty(\Omega;\G)$ such that
$G_n k_n \in \SBV(\Omega; \On)$, and
$|D(G_n k_n)|(\Omega)\leq C$, so that the functions
	\begin{equation}\label{2602241348}
		\widehat{\beta}^\rho_n:=G_n k_n h_n^{\frac{1}{2}} \chi_{\Omega\sm F_n^\rho}+\Id \chi_{F_n^\rho}=
		G_n \beta_n \chi_{\Omega\sm F_n^\rho}+\Id \chi_{F_n^\rho}
	\end{equation}
	belong to  $\SBV(\Omega; \Mno)$ with $\mathrm{det}\,\widehat{\beta}_n(x) \neq 0$ 
	for a.e.\ $x \in \Omega$,  and have equibounded $\BV$ norms.
	Hence 
	there exists a function $\beta^\rho\in \BV(\Omega; \Mno)$ such that, up to a subsequence,
	$(\widehat{\beta}_n) $ converge to $\beta^\rho$ in the weak* topology of $\BV(\Omega; \Mno)$.
Recalling \eqref{eqcptpropenrhoF}, we infer that
\begin{equation}\label{f: endc2}
		G_n \beta_n \chi_{\Omega\sm F_n^\rho}+\beta_n  \chi_{ F_n^\rho}\to \beta^\rho \quad\text{in }L^1(\Omega; \Mno),
	\end{equation}
	so that the functions $G_n$ as in the statement may be taken equal to $\Id$ in $F_n$. Notice that $G_n \beta_n \in \SBVG(\Omega; \Mno)$ by \eqref{2602241348} and \eqref{0305251910}. 
	Finally, since by \eqref{0305251910}
	\[
	\frac{1}{\delta_{\varepsilon_n}} \int_A \dist^2(\beta_n, \On) \dx \leq K,
	\]
	we obtain that  $\beta^\rho\in \BV(\Omega; \On)$.

	As in the final part of the proof of Theorem~\ref{t:comp1}, we can ensure by a diagonal argument that $\beta^\rho=:\beta$ does not depend on $\rho$, and that $\beta \in \SBV_0(\Omega;\On)$.
	We notice that one could obtain the fact that $\beta \in SBV_0(\Omega;\On)$ also arguing directly by slicing.
\end{proof}

\section{The surface energy density}\label{sec:surfen}

The limit energy is defined on functions with pure jump part, namely in $\PC$. In this section we give the definition of the surface energy density and the proof of its main properties. 

The surface energy density is expressed by a cell formula in one dimension. It depends on a parameter $\lambda \in [0,+\infty]$ determined by the asymptotic ratio between $\varepsilon$ and $\delta_\varepsilon$ in the definition of the approximating functionals, see~\eqref{1002241910} in  Section~\ref{sec:intro}.
In order to prove its main properties, such as the $s |\log s|$ growth for small opening (see Proposition~\ref{prop:proprietag}), it is convenient to recast the energy density through a 1d cell formula which is invariant by time reparametrization (Proposition~\ref{le:equivglambda}).

We conclude the section with a technical result, Lemma~\ref{lemmalambdainfty}, used for the proof of the upper bound when $\lambda=+\infty$.


\smallskip

Let $I_T:=(-T,T)$ for $T>0$,
\[
F^T(\beta,v):= \int_{-T}^T \left( f^2(v) |\beta'|^2 + \frac{(1-v)^2}{4} + |v'|^2 \right)\dt, \qquad 
(\beta,v)\in H^1(I_T;\Mno\times [0,1]),
\] 
and
\[
F^T_\lambda(\beta,v):= F^T(\beta,v)+\lambda \int_{-T}^T\dist^2(\beta, \On) \dt \quad \text{for } \lambda\in [0,+\infty).
\]
For $ R^-$, $R^+\in \On$ we let
\begin{equation}\label{eqdefcalUT}
\calU^T[ R^-,R^+]:=  \{  ( \beta,v)  \in  H^1(I_T; \Mno \times [0,1]) \colon
		\beta(-T)= R^-,\, \beta(T)=R^+,\, v(\pm T)=1 \}
\end{equation}
be the set of admissible paths connecting $ R^-$ to $R^+$ in $I_T$, and with fixed boundary conditions, $v=1$, for the damage coefficient.
For $\lambda \in [0,+\infty)$ we set
\begin{equation}\label{glambda}
g^*_\lambda( R^-, R^+):=
		\lim_{T\uparrow +\infty}  \inf\left\{F^T_\lambda(\beta,v)
		\colon (\beta,v)\in \calU^T[ R^-,R^+] \right\},
\end{equation}
and then
\begin{equation}\label{eqdefgstar}
g^*_\infty( R^-, R^+):=\sup_{\lambda>0}
g^*_\lambda( R^-, R^+).
\end{equation}
One easily checks that the limit in \eqref{glambda} exists, and it equals the infimum over all $T>0$.  In order to recover the invariance under the action of the group we then define, for $\lambda\in[0,\infty]$,
\begin{equation}\label{glambdacar}
    g_\lambda( R^-, R^+):=	\min_{G\in\G}g_\lambda^*( R^-, GR^+).
\end{equation}
Since the group is finite, in the definition of $g_\infty$ the minimum over $\G$ and the supremum over $\lambda$ commute:
\begin{equation}\label{glambdainfty}
    g_\infty( R^-, R^+):=	\min_{G\in\G}
    g_\infty^*( R^-, GR^+)
    =\sup_{\lambda>0}
    g_\lambda( R^-, R^+).
\end{equation}

We first prove some basic properties of $g_\lambda$. 
\begin{proposition}\label{p:bpg}
For every $\lambda\in [0,+\infty]$,  $R,\,  R^-,\, R^+ \in \On$, $G\in \G$ it holds that 
\begin{enumerate}
	\item\label{p:bpg:0} $g_\lambda(R, GR)=0$;
	\item\label{p:bpg:G} $g_\lambda( R^-, R^+)=g_\lambda( R^-, GR^+)$;
	\item\label{p:bpg:Id} $g_\lambda( R^-, R^+)=g_\lambda( R^-(R^+)^T, \Id)$;
\item\label{p:bpg:sym} $g_\lambda( R^-, R^+)=g_\lambda(R^+, R^-)$;
\item\label{p:bpg:subadd} $g_\lambda( R^-,R^+)
    \le g_\lambda(R^-, R)+g_\lambda(R,R^+)$, and the same for $g^*$.
\end{enumerate}
\end{proposition}
\begin{proof} 
Properties \ref{p:bpg:0} and \ref{p:bpg:G} follow immediately from
\eqref{glambdacar} and the fact that $F^T(\beta,1)$ vanishes if $\beta$ is a constant rotation.

Assertion \ref{p:bpg:Id} follows from the fact that for every
$(\beta,v)\in \calU^T(R^-,R^+)$, setting
 $\widetilde{\beta}(t):=\beta(t)(R^+)^T$
 one has $(\widetilde\beta,v)\in \calU^T[R^-(R^+)^T,\Id]$
 and, since $R^+\in \On$,
$| \widetilde{\beta}'|=|\beta'|$, and $\dist(\widetilde{\beta}, \On)=\dist(\beta, \On)$, so that
$F^T_\lambda(\beta,v)=F^T_\lambda(\widetilde\beta,v)$.

A similar argument is valid for $\widehat{\beta}(t):={\beta}(-t)$ , proving that \ref{p:bpg:sym} holds.

To prove \ref{p:bpg:subadd}, for any $T>0$ and any $(\beta_1,v_1)\in \calU^T[ R^-,R]$ and
$(\beta_2,v_2)\in \calU^T[ R,R^+]$,
we define the pair
\begin{equation*}
(\beta,v)(t):=
\begin{cases}
 (\beta_1,v_1)(t+T), & \text{ if } t\le 0,\\
 (\beta_2,v_2)(t-T), & \text{ if } t> 0.
\end{cases}
\end{equation*}
One checks that
$(\beta,v)\in \calU^{2T}[ R^-,R^+]$, and taking the infimum in $(\beta_1,v_1)$ and $(\beta_2,v_2)$, and the limit $T\to\infty$, proves the assertion for $g_\lambda^*$. Taking the minimum over $\G$ gives the result for $g_\lambda$.
\end{proof}

In order to prove the main properties of $g_\lambda$ it is conveninent to deal with a different characterization, obtained through minimization of profiles invariant by time-reparametrization. For
fixed $ R^-$, $R^+\in \On$, we define 
\begin{equation}\label{f:admiss2}
\begin{split}
	\calU[ R^-,R^+]:=  \{  ( & \beta,v)  \in  H^1((0,1); \Mno \times [0,1])\colon \\ &
	 \beta(0)= R^-,\, \beta(1)=R^+,\, v(0)=v(1)=1 \},
\end{split}
\end{equation}
and 
\begin{equation}\label{gbarla}
\begin{split}
	\ol g_\lambda( R^-, R^+) :=
	\inf \bigg\{  \int_0^1 & \sqrt{(1-v)^2 + 4\lambda \, \dist^2(\beta, \On)} \sqrt{f^2(v) |\beta'|^2+|v'|^2} \dt\colon \\ &
	(\beta,v) \in \calU[ R^-,R^+]  \bigg\}
\end{split}
\end{equation}
for $\lambda \in [0,+\infty)$, and 
\begin{equation}\label{gbarinf}
\ol g_\infty( R^-, R^+) :=
	\sup_{\lambda>0}
\ol g_\lambda( R^-, R^+).
\end{equation}

\begin{proposition}\label{le:equivglambda}
For every $\lambda \in[0,+\infty]$, it holds that $g^*_\lambda=\ol g_\lambda$.
\end{proposition}
\begin{proof} 
The proof follows the same lines of the one of \cite[Proposition~4.3]{CFI14}. We argue for $\lambda \in [0,+\infty)$, the case $\lambda=+\infty$ being an easy consequence using the definitions.

The fact that $g^*_\lambda\geq \ol g_\lambda$ follows by a direct application of Young's inequality  and the  1--homogeneity in the derivatives of the integral functional in \eqref{gbarla}. Namely, for every $(\beta,v)\in
\calU^T[ R^-,R^+]$ we have
\[
\begin{split}
	F^T_\lambda(\beta,v) & \geq
	\int_{-T}^T \sqrt{(1-v)^2 + 4\lambda \, \dist^2(\beta, \On)} \sqrt{f^2(v) |\beta'|^2+|v'|^2} \dt \\ & =
	\int_0^1 \sqrt{(1-\widetilde v)^2 + 4\lambda \, \dist^2(\widetilde \beta, \On)} \sqrt{f^2(v) |\widetilde \beta'|^2+|\widetilde v'|^2} \dt \ge\ol g_\lambda(R^-,R^+)
\end{split}
\] 
for $\widetilde \beta (t)= \beta(2Tt-T)$, $\widetilde v (t)=v(2Tt-T)$, $t\in (0,1)$, and the inequality follows from the definition of $g^*_\lambda$.

To prove the opposite inequality, we first remark  that, by an approximation argument through more regular $(\beta,v)$, 
for any $R^-$, $R^+$, $\eta\in(0,1)$ we can find
competitors $(\beta,v)$ belonging to $ W^{1,\infty}((0,1); \Mno \times [0,1])
\cap \calU[ R^-,R^+]
$ and satisfying
\begin{equation}\label{1902241028}
	\int_0^1\sqrt{(1-v)^2 + 4\lambda \, \dist^2(\beta, \On)} \sqrt{f^2(v) |\beta'|^2+|v'|^2} \dt < \ol g_\lambda(  R^-, R^+) + \eta.
\end{equation}
We set
\begin{equation}\label{1902241041}
v_\eta(t):=v(t) \wedge (1-\eta), \quad
	\psi_\eta(t):=2 \int_0^t  \sqrt{\frac{\eta + f^2(v_\eta) |\beta'|^2+|v_\eta'|^2}{(1-v_\eta)^2 + 4 \lambda \, \dist^2(\beta, \On)}}  \,\mathrm{d}s, \qquad t\in[0,1],
\end{equation} 
obtaining a bilipschitz map $\psi_\eta\colon [0,1] \to [0, \psi_\eta(1)]$. Then we choose competitors for $g^*_\lambda( R^-, R^+)$ as follows:
\[
\widehat{\beta}_\eta (t):=
\begin{cases}
	 R^-, & t\in [-1-\psi_\eta(1),0], \\
	\beta (\psi_\eta^{-1} (t)) & t\in[0, \psi_\eta(1)], \\
	R^+, & t\in [\psi_\eta(1), 1+\psi_\eta(1)];
\end{cases}
\widehat{v}_\eta:=
\begin{cases}
	1, & t\in [-1-\psi_\eta(1),-1], \\
	(1-\eta)-t\eta, & t\in [-1,0], \\
	v_\eta (\psi_\eta^{-1} (t)), & t\in [0, \psi_\eta(1)], \\
	1+(t-\psi_\eta(1)-1)\eta, & t\in [\psi_\eta(1), 1+\psi_\eta(1)].
\end{cases}
\]
Namely, setting $T_\eta:= \psi_\eta(1)  +1$, we have that $(\widehat{\beta}_\eta, \widehat{v}_\eta)\in \calU^{T_\eta}[ R^-, R^+]$, and hence
\begin{equation}\label{f:bpg1}
	g^*_\lambda( R^-,  R^+) \leq F^{T_\eta}_\lambda(\widehat{\beta}_\eta, \widehat{v}_\eta),
\end{equation}
due to the fact that the optimization problems defining $g^*_\lambda( R^-,  R^+) $ are decreasing in $T$.
We  now prove that
\begin{equation*}\label{f:bpg2}
	\begin{split}
F^{T_\eta}_\lambda (\widehat{\beta}_\eta, \widehat{v}_\eta)
		&<  \int_0^1 \sqrt{(1-v)^2 + 4\lambda \, \dist^2(\beta, \On)} \sqrt{f^2(v) |\beta'|^2+|v'|^2} \dt+o_{\eta\to 0}(1).
	\end{split}
\end{equation*}
Using the change of variable $t=\psi_\eta(s)$, it is straightforward that
\[
\begin{split}
	\int_{0}^{\psi_\eta(1)} & \left(f^2(\widehat{v}_\eta) |\widehat{\beta}_\eta'|^2  + |\widehat{v}_\eta'|^2\right)\dt + \int_{0}^{\psi_\eta(1)}  \left(\frac{(1-\widehat{v}_\eta)^2}{4} +\lambda \,\dist^2(\widehat{\beta}_\eta, \On) \right) \dt 
	\\ &
	= \int_{0}^{1}  \left(f^2({v}_\eta) |\beta'|^2  + |{v}_\eta'|^2\right)\frac{1}{\psi_\eta'}\ds  + 
	\int_{0}^{1}  \left(\frac{(1-{v}_\eta)^2}{4} +\lambda \,\dist^2(\beta, \On) \right) \psi_\eta'\ds
	\\ &
	= \frac{1}{2}\int_{0}^{1} \left(f^2({v}_\eta) |\beta'|^2  + |{v}_\eta'|^2\right)
	\sqrt{\frac{(1-v_\eta)^2 + 4 \lambda \, \dist^2(\beta, \On)}{\eta + f^2(v_\eta) |\beta'|^2+|v_\eta'|^2}}  
	\ds
	\\ & 
	+  \frac{1}{2} \int_{0}^{1}  \sqrt{\eta + f^2(v_\eta) |\beta'|^2+|v_\eta'|^2}\,  \sqrt{(1-v_\eta)^2 + 4 \lambda \, \dist^2(\beta, \On)}\ds
	\\ &
	\leq c_\lambda \frac{\sqrt{\eta}}{2}
	 +  \int_{0}^{1}  \sqrt{f^2(v_\eta) |\beta'|^2+|v_\eta'|^2}\, \sqrt{(1-v_\eta)^2 + 4 \lambda \, \dist^2(\beta, \On)}\ds,
\end{split}
\]
where $c_\lambda=\sqrt{1+8\lambda d}$
(we used here subadditivity of the square root and the bound $\dist(\beta,\On)\le 2\sqrt d$).
Denoting $E_\eta:= I_{T_\eta}  \setminus [0, \psi_\eta(1)]$, we have
\[
\begin{split}
	\int_{E_\eta} & \left(f^2(\widehat{v}_\eta) |\widehat{\beta}_\eta'|^2  + |\widehat{v}_\eta'|^2 + \frac{(1-\widehat{v}_\eta)^2}{4} +\lambda \,\dist^2(\widehat{\beta}_\eta, \On) \right) \dt 
	\\ &
	= \int_{E_\eta}  \left(|\widehat{v}_\eta'|^2 + \frac{(1-\widehat{v}_\eta)^2}{4}  \right) \dt = 2 \eta^2 + \frac{\eta^2}{6},
\end{split}
\]
and hence 
\[
\begin{split}
	F^{T_\eta}_\lambda(\widehat{\beta}_\eta, \widehat{v}_\eta) &
	\leq c_\lambda \sqrt{\eta}+3\eta^2+ \int_{0}^{1}  \sqrt{f^2(v_\eta) |\beta'|^2+|v_\eta'|^2}\, \sqrt{(1-v_\eta)^2 + 4 \lambda \, \dist^2(\beta, \On)}\dt.	 
\end{split}
\]
Since $v_\eta'=0$ almost everywhere in the set $\{v_\eta \neq v\}$, we have
\[
\begin{split}
	\int_{0}^{1} & \sqrt{f^2(v_\eta) |\beta'|^2+|v_\eta'|^2}\, \sqrt{(1-v_\eta)^2 + 4 \lambda \, \dist^2(\beta, \On)}\dt
	\\  =&
	\int_{\{v_\eta=v\}}	  \sqrt{f^2(v) |\beta'|^2+|v'|^2}\, \sqrt{(1-v)^2 + 4 \lambda \, \dist^2(\beta, \On)}\dt \\ &
	+ \int_{\{v_\eta \neq v\}}	f(1-\eta)\,  |\beta'|\, \sqrt{\eta^2 + 4 \lambda \, \dist^2(\beta, \On)}\dt.
\end{split}
\]  
Finally, let $\ol\omega(\eta)$ be the continuity modulus of $\dfrac{1-s}{|\log(1-s)|}f(s)$ near $1$, so that
\begin{equation}\label{eq-approx-f}
\dfrac{|\log(1-s)|}{1-s} (\ell-\ol \omega(\eta)) \leq f(s) \leq \dfrac{|\log(1-s)|}{1-s} (\ell+\ol \omega(\eta)) , \qquad s\in [1-\eta,1).
\end{equation}

Since for every $a\geq 0$ the function 
$
s \mapsto \dfrac{|\log(1-s)|}{1-s} \sqrt{(1-s)^2+a}=
|\log(1-s)| \sqrt{1+ a(1-s)^{-2}}$
is increasing in $[0,1)$, and $v\geq 1-\eta$ in the set $\{v_\eta \neq v\}$, we obtain that
\begin{equation}
	\begin{split}\label{eq-estimate-truncation}
		\int_{\{v_\eta \neq v\}} &	f(1-\eta)\,  |\beta'|\, \sqrt{\eta^2 + 4 \lambda \, \dist^2(\beta, \On)}\dt 
		\\ & \leq 
		\frac{\ell+\ol \omega(\eta)}{\ell-\ol \omega(\eta)}\int_{\{v_\eta \neq v\}}	f(v)\,  |\beta'|\, \sqrt{(1-v)^2 + 4 \lambda \, \dist^2(\beta, \On)}\dt
		\\ & \leq 
		\frac{\ell+\ol \omega(\eta)}{\ell-\ol \omega(\eta)}\int_{\{v_\eta \neq v\}}	 \sqrt{f^2(v) |\beta'|^2+|v'|^2}\, \sqrt{(1-v)^2 + 4 \lambda \, \dist^2(\beta, \On)}\dt.
	\end{split}
\end{equation}
Summarizing,  gathering \eqref{f:bpg1} and the above computations we obtain that
\begin{equation}\label{eqgstarolgfinal}
\begin{split}
g^*_\lambda & ( R^-,  R^+)  \leq \ol g_\lambda(  R^-,  R^+)  \\ &
+ c_\lambda \sqrt{\eta}+4\eta^2+
\frac{2\ol \omega(\eta)}{\ell-\ol \omega(\eta)} \int_{0}^{1}  \sqrt{f^2(v) |\beta'|^2+|v'|^2}\, \sqrt{(1-v)^2 + 4 \lambda \, \dist^2(\beta, \On)}\dt.
\end{split}
\end{equation}
This permits us to conclude by the arbitrariness of $\eta>0$.
\end{proof}

At various steps hereafter we shall need to transport trajectories onto $\On$. We present here the basic computation to estimate the length of the resulting curves.
\begin{lemma}\label{lemmaprojOn}
 Let $I\subset\R$ be an interval, $\gamma\in W^{1,1}(I;\Mno)$ be such that
 $\dist(\gamma(t),\On)\le 1/2$ almost everywhere.
 Then there exist $C>0$ independent of $\gamma$, and  $\widehat\gamma\in W^{1,1}(I;\On)$
 such that
 \begin{equation}
  |\widehat\gamma'|(t)\le (1+C\dist(\gamma(t),\On)) |\gamma'|(t)\ \text{for a.e.}\ t\in I,\hskip5mm
  \text{ and }\hskip5mm
  \widehat\gamma=\gamma \text{ on } \{\gamma\in \On\}.
 \end{equation}
\end{lemma}
\begin{proof}
We use, as in the argument leading to \eqref{2801241031}, that $\psi(A):=A(A^TA)^{-1/2}$ is a smooth map from $\Mno\setminus\{\det=0\}$ to $\On$, which is in particular Lipschitz on the set
$\{\dist(\cdot,\On)\le 1/2\}$.
At the same time, for $R\in \On$
one has $\psi(R)=R$, and for any
$|\eta|\le 1/2$, with $\eta\in \Mno$, a Taylor expansion gives
\begin{equation}\begin{split}
 \psi(R+\eta)=&
 (R+\eta)\left((R^T+\eta^T)(R+\eta)\right)^{-1/2}
 =(R+\eta)(\Id+\eta^T R+ R^T\eta+\eta^T\eta)^{-1/2}\\
 =&R+\eta-\frac12 R\eta^T R-\frac12 \eta+O(|\eta|^2),
\end{split}\end{equation}
so that $|D\psi(R)\eta|\le |\eta|$, and hence $|D\psi(R)|_\mathrm{op}\leq 1$, which by regularity of $\psi$ implies
$|D\psi(\xi)|_\mathrm{op}\le 1+C\dist(\xi,\On)$,
whenever $\dist(\xi,\On)\le 1/2$.
Let $\widehat\gamma:=\psi\circ\gamma$.
Then $|\widehat\gamma'(t)|\le
|D\psi(\gamma(t))|_\mathrm{op}|\gamma'(t)|$ and the proof is concluded.
\end{proof}

We are now in a position to prove the main properties of $g_\lambda$, collected in the following result.

\begin{proposition}\label{prop:proprietag}
For every $\lambda\in [0,+\infty]$, the function $ g^*_\lambda \colon \On \times \On \to [0,+\infty)$ satisfies the following properties:
\begin{enumerate}
\item\label{prop:proprietag:upperbound}  there exists $\kappa \geq 1$ depending only on $d$ and $f$ and such that
	\begin{equation}\label{0105232216}
		g^*_\lambda( R^-, R^+) \leq 1 \wedge  \kappa  \Big(\, | R^+- R^-| \,\big|\log | R^+ - R^-|\big| +| R^+ - R^-|\Big),
	\end{equation}
	for every $ R^+,  R^- \in \On$, $\lambda\in [0,\infty]$;

\item\label{prop:proprietag:RS}  for every $R\in\On$ and for every sequence   $(R_n)_n \subset \On $  converging to  $R$, with $R_n \neq R$, there exists
	\[\lim_{n \to +\infty} \frac{g^*_\lambda(R_n, R)}{|R_n-R| \,\big|\log |R_n-R|\big|}=\frac{\ell}{2};
	\]
\item\label{prop:proprietag:C0}  $g^*_\lambda \in C^0(\On\times \On)$.
\end{enumerate}
All properties are immediately inherited by $g_\lambda$.
\end{proposition}
\begin{proof} 
\noindent \textit{Proof of \ref{prop:proprietag:upperbound}.}
It suffices to prove the bound for $\lambda\in [0,\infty)$, with $\kappa$ not depending on $\lambda$.
For every $ R^-$, $ R^+\in\On$ we set
\[
\ol\beta(t):=
\begin{dcases}
	 R^-, & t\in \left[0,\frac13\right], \\
R^-+(R^+-R^-)(3t-1), &   t\in \left[\frac13,\frac23\right], \\
	 R^+,&  t\in \left[\frac23,1\right];
\end{dcases}
\qquad 
\ol v(t) :=
\begin{dcases}
	1-3t, &  t\in \left[0,\frac13\right], \\
	0, &   t\in  \left[\frac13,\frac23\right], \\
	3t-2, &  t\in  \left[\frac23,1\right] .
\end{dcases}
\]	
Applying Proposition \ref{le:equivglambda}, and by a direct computation, we obtain
\begin{equation}\label{eqglambdale1}
\begin{split}
	g^*_\lambda( R^-, R^+) =& \ol g_\lambda( R^-, R^+) 
	\leq   \int_0^1 \sqrt{(1-\ol v)^2+4\lambda\, \dist^2(\ol\beta,\On)} \sqrt{f^2(\ol v) |\ol\beta'|^2+|\ol v'|^2} \dt\\
	= & \int_{(0,1/3)\cup (2/3,1)}(1-\ol v) |\ol v'| \dt= 1.
\end{split}
\end{equation}
This completes the proof in the case
$|R^+-R^-|\ge 1$.

Assume now
$|R^+-R^-|< 1$.
We shall complete the proof of \eqref{0105232216} using again the fact that $g^*_\lambda=\ol g_\lambda$, and choosing properly the competitors in the optimization problem defining $\ol g_\lambda$.
Using Lemma~\ref{lemmaprojOn} on the affine map joining $R^-$ with $R^+$, whose image is contained in a neighbouhood of $\On$ of size $|R^--R^+|/2<1/2$, we produce a curve 
$\gamma\in C^1([\frac13, \frac23];\On)$ such that
$\gamma(\frac13)=R^-$, $\gamma(\frac23)=R^+$, and
\[
\int_{\frac13}^{\frac23} |\gamma'(t)|\dt\le  |R^--R^+|+C|R^--R^+|^2.
\]

For $s\in (0,1)$ chosen below, we set
\begin{equation*}
	\beta(t):=\begin{dcases}
		 R^-, \quad & t\in \left[0, \frac13\right],\\
		\gamma(t),\quad & t\in \left[\frac13, \frac23\right],\\
		 R^+, \quad & t\in \left[\frac23,1\right],
	\end{dcases} \qquad
	v(t):=\begin{dcases}
		1-3t s^{1/2}, \quad & t\in \left[0, \frac13\right],\\
		1- s^{1/2}, \quad & t\in \left[\frac13, \frac23\right],\\
		1 + 3 (t-1)  s^{1/2}, \quad & t\in \left[\frac23,1\right].
	\end{dcases}
\end{equation*}
An explicit computation, similar to \eqref{eqglambdale1}, gives
\[
\begin{split}
	\ol g_\lambda( R^-,  R^+)  & \leq \int_{(0,1/3)\cup(2/3,1)}(1-v) |v'| 
	\dt + \int_\frac13^\frac23 s^{1/2} f(1- s^{1/2}) |\gamma'(t)|\, \dt
	\\  
	& =
	s +(|R^--  R^+|+C|R^--  R^+|^2) s^{1/2} f(1-s^{1/2})
\end{split}
\]
and hence choosing $s=|R^--  R^+|$ gives
\begin{equation}\label{1902241215}
	\ol	g_\lambda( R^-,R^+) \leq  | R^- - R^+| +|R^--  R^+|^{3/2}(1+C|R^--  R^+|)  f(1-| R^- - R^+|^{1/2}) .
\end{equation}
The conclusion then follows recalling  
\eqref{0405230818}
and Proposition~\ref{le:equivglambda}.

\smallskip

\noindent \textit{Proof of \ref{prop:proprietag:RS}.} Let $(R_n)_n\subset \On$ converging to $R$ be given. Notice that $R_n$, $R$ belong to the same connected component of $\On$, say $\SOn$, namely $R_n \in \SOn$ for every $n\in\N$, $R\in \SOn$.   We start by proving the lower bound:
\begin{equation}\label{0205230936}
	\liminf_{n\to +\infty}\frac{g_\lambda(R_n, R)}{|R_n-R| \,\big|\log |R_n-R|\big|}\geq \frac{\ell}{2}.
\end{equation}
By monotonicity, it suffices to prove it for $\lambda=0$.
If \eqref{0205230936} does not hold, then there are $\delta>0$ and 
$R_n\in \SOn$, with $s_n:=|R_n-R|\to0$,
and
$(\ol\beta_n,v_n)\in\calU[R_n,R]$ such that
\begin{equation}\label{0205230936b}
\int_0^1 (1-v_n) \sqrt{f^2(v_n) |\ol\beta_n'|^2+|v_n'|^2} \dt \le
 \frac{\ell-\delta}{2} s_n |\log s_n|.
\end{equation}
For every $n$, since
$v_n(0)=v_n(1)=1$ one obtains
\begin{equation}\label{0205231201}
{(1-v_{n}(t))^2} \le
\int_0^1  (1-v_n)|v_n'|\dt \le
 \frac{\ell}{2} s_n |\log s_n|, \qquad t\in (0,1).
\end{equation}
Moreover, using \eqref{0205230936b} and
$\int_0^1  |\ol\beta_n'| \dt \geq s_n$, we get
\begin{equation}\label{eqmin1vnfvn}
 \begin{split}
 \frac{\ell-\delta}{2} s_n |\log s_n|\ge&
\int_0^1 (1-v_n) f(v_n) |\ol\beta_n'|\dt
\ge s_n \min_{[0,1]} (1-v_n) f(v_n).
\end{split}
\end{equation}
Using first \eqref{eqmin1vnfvn} and then
\eqref{0205231201},
\begin{equation}
\begin{split}
 \frac{\ell-\delta}{2} \ge&
 \min_{[0,1]}
\frac{(1-v_n) f(v_n)}{|\log(1-v_n)|}
\frac{|\log(1-v_n)|} {|\log s_n|}
\ge\min_{[0,1]}
\frac{(1-v_n) f(v_n)}
{|\log(1-v_n)|}
\frac{{|\log (\frac\ell2 s_n |\log s_n|)|}}{ 2|\log s_n|}.
\end{split}\end{equation}
Now we take the limit $n\to\infty$.
By \eqref{0205231201},
$v_n\to1$ uniformly. Therefore
(recalling \eqref{0405230818})
the first fraction converges to $\ell$.
At the same time, since $s_n\to0$ the second fraction converges to $1/2$.
This gives a contradiction, since we assumed $\delta>0$.
This proves
\eqref{0205230936} and then the lower bound.

The upper bound follows immediately using first \eqref{1902241215} and then
\eqref{0405230818}. Precisely, with $s_n:=|R_n-R|\to0$,
\[
\begin{split}
	\limsup_{n\to +\infty}  \frac{g_\lambda( R_n, R)}{|R_n-R| \,\big|\log |R_n-R|\big|}
	\leq&
	\limsup_{n\to +\infty}\frac{s_n +
	(1+Cs_n)
	s_n^{3/2} f(1-s_n^{1/2 })}
	{s_n \,\big|\log s_n\big|}
	\\  =&
	\limsup_{n\to +\infty}\frac{s_n^{1/2}
	f(1-s_n^{1/2})}
	{2\big|\log s_n^{1/2}\big|} =
	\frac{\ell}{2}.
\end{split}
\]

Since we used only the bound in
\eqref{1902241215}, which does not depend on $\lambda$, the limit is uniform in $\lambda$, and therefore \ref{prop:proprietag:RS} holds also for $\lambda=\infty$.
\medskip

\noindent \textit{Proof of \ref{prop:proprietag:C0}.} This follows by combining \eqref{0105232216}, which guarantees the continuity at the origin, and the subadditivity stated in
Proposition~\ref{p:bpg}\ref{p:bpg:subadd}.
\end{proof}

The case $\lambda=\infty$ requires a separate treatment for construction of a recovery sequence leading to the upper bound (see Section \ref{Subsec:limsup}). If $R^+$ and $R^-$ have the same determinant, one can find admissible paths joining them that stay within $\On$. If not, however, such curves do not exist, therefore a different construction is needed. The approximating admissible paths are actually simpler, and can be constructed explicitly, however they depend on $\delta_\eps$, therefore in the global construction one cannot take the limit $\eps\to0$ before the limit $T\to\infty$. The next Lemma presents this construction.
\begin{lemma}\label{lemmalambdainfty}
Assume $R^+, R^-\in\On$.
\begin{enumerate}
\item\label{lemmalambdainftyginfty} If $\det R^+\ne \det R^-$ then
there is no
  $(\beta,v)\in \calU^T[R^-,R^+]$ such that
  $\beta(t)\in \On$ almost everywhere.
 \item\label{lemmalambdainftyson}
 If $g_\infty^*(R^-,R^+)<1$ then
  $\det R^+=\det R^-$ and  there are sequences
  $(\beta_T,v_T)\in \calU^T[R^-,R^+]$ such that
  $\beta_T(t)\in \On$ almost everywhere,
  \begin{equation}
\lim_{T\to\infty} F^T(\beta_T, v_T)\leq    g_\infty^*(R^-,R^+),
  \end{equation}
  and \begin{equation}
           \|\beta'_T\|_{L^2((-T,T))}^2\le C(R^+, R^-).
          \end{equation}
  \item\label{lemmalambdainftyub}
  For any $T,\lambda \ge 1$  there exists $(\beta_{T,\lambda}, v_{T,\lambda})\in\calU^T[R^-,R^+]$ such that
  \begin{equation}\label{eqseqforub}
F^T_{\lambda}(\beta_{T,\lambda}, v_{T,\lambda}) 
   \le g_\infty^*(R^-,R^+) +\omega(T)+\omega(\lambda)
  \end{equation}
and 
\begin{equation}\label{eql2boundbetatla}
 \|\beta_{T,\lambda}'\|_{L^2((-T,T))}^2 \le C\lambda T,
\end{equation}
  with $\omega(s)\to 0$ as $s\to +\infty$.
\end{enumerate}
 \end{lemma}
In the upper bound for $\lambda=\infty$ we shall use \ref{lemmalambdainftyub} with sequences $\lambda_\eps=\eps/\delta_\eps$ and
$T_\eps\to +\infty$.
One can easily also obtain that \eqref{eqseqforub} is an equality, this is however not needed.

\begin{proof}
\emph{Proof of \ref{lemmalambdainftyginfty}.}
Since $\beta\in H^1$, it has a continuous representative, and so does $\det\beta$. A continuous function which takes values in $\{\pm1\}$ is constant.

\emph{Proof of \ref{lemmalambdainftyson}.}
For every $\lambda\ge1$ we pick $T_\lambda$ and then $(\beta_\lambda, v_\lambda)\in \calU^{T_\lambda}[R^-,R^+]$ so that
\begin{equation}\label{eqlainftyftlabv}
 F^{T_\lambda}_\lambda (\beta_\lambda,v_\lambda)\le g^*_\lambda(R^-,R^+) +\frac1\lambda.
\end{equation}
We can assume that $\beta_\lambda$ and $v_\lambda$ are continuous.
We estimate, using that $v_\lambda(\pm T_\lambda)=1$ and the fundamental theorem of calculus on $(-T_\lambda,t)$ and on $(t,T_\lambda)$,
\begin{equation}\label{eqivlader}
\max_{I_{T_\lambda}} (1-v_\lambda)^2
\le \int_{-T_\lambda}^{T_\lambda} |1-v_\lambda| \, |v_\lambda'| \dt \le
g^*_\lambda(R^-,R^+) +\frac1\lambda.
\end{equation}
The right-hand side converges to
$g^*_\infty(R^-,R^+)<1$ for $\lambda\to\infty$.
Setting $\delta:=\frac14(1-g^*_\infty(R^-,R^+))>0$, for $\lambda$ large we have
\begin{equation}\label{eqivladerris}
(1-v_\lambda(t))^2\le 1-2\delta \hskip1cm \text{ for all $t\in I_{T_\lambda}$}.
\end{equation}
By monotonicity of $f$, this implies $f(v_\lambda)\ge f(\delta)$ everywhere.
Therefore  \eqref{eqlainftyftlabv} implies
\begin{equation}\label{f:h1bounddd}
 \int_{-T_\lambda}^{T_\lambda} |\beta_\lambda'|^2 +|v_\lambda'|^2 \dt \le C,
 \end{equation}
where $C>0$ may depend on
$R^-$, $R^+$ via $\delta$, but not on $\lambda$.
 By the same computation as in \eqref{eqivlader},  
using that $\dist(\beta_\lambda(\pm T_\lambda),\On)=0$,
 \begin{equation}\label{eqivlader2}
	\begin{split}
2  \lambda^{1/2}f(\delta)\max_{t\in I_{T_\lambda}} \dist^{2}(\beta_\lambda(t), \On) &
\le \int_{-T_\lambda}^{T_\lambda}
2\lambda^{1/2} f(v_\lambda)\, \dist(\beta_\lambda, \On)|\beta_\lambda'| \dt 
\\ &
\le
g^*_\lambda(R^-,R^+) +\frac1\lambda +C \le
C+1,
\end{split}
\end{equation}
therefore $\|\dist(\beta_\lambda, \On)\|_{L^\infty}\le C \lambda^{-1/4}$ (recall that $C$ depends on $\delta$ but not on $\lambda$).

Let $\widehat\beta_\lambda$
 be the result of Lemma~\ref{lemmaprojOn} applied to $\beta_\lambda$, which obeys
 \begin{equation}
 |\widehat\beta_\lambda'| \le (1 + C \dist(\beta_\lambda,\On))|\beta_\lambda'|.
\end{equation}
Therefore, by \eqref{f:h1bounddd},  $\widehat\beta_\lambda \in H^1(I_{T_\lambda};\On)$, and hence, by \ref{lemmalambdainftyginfty}  $\det R^+=\det R^-$. Moreover, it holds
\begin{equation}
  F^{T_\lambda}_\lambda (\widehat \beta_\lambda,v_\lambda)\le
  \left(1+\frac{C}{\lambda^{1/4}}\right)^2
  F^{T_\lambda}_\lambda (\beta_\lambda,v_\lambda)\le
  g^*_\lambda(R^-,R^+) +\frac C{\lambda^{1/4}}+\frac1\lambda,
  \end{equation}
  and, taking $\lambda\to\infty$, we obtain
  \begin{equation}
\limsup _{\lambda\to\infty} F^{T_\lambda}_\lambda (\widehat \beta_\lambda,v_\lambda)\le
g^*_\infty(R^-,R^+).
  \end{equation}
By monotonicity in $T$, the same holds along any sequence $T\to\infty$.

\emph{Proof of \ref{lemmalambdainftyub}:}
If $g^*_\infty(R^-,R^+)<1$, it follows from \ref{lemmalambdainftyson} with $(\beta_{T,\lambda},v_{T,\lambda})=(\beta_T,v_T)$.

If not, we construct explicitly an appropriate sequence
 $(\beta_{T,\lambda},v_{T,\lambda})\in \calU^{T}[R^-,R^+]$, for any $T>4$.
Fix $\eta\in (0,1)$, chosen below depending on $\lambda$ and $T$, and set
\begin{equation}
 v_{T,\lambda}(t):=
 \begin{cases}
 0, &\text{ if } |t|<\eta,\\
 1-e^{-(|t|-\eta)/2},
 &\text{ if } \eta\le |t|< T-2+\eta,\\
 1-e^{-(T-2)/2}
 (T-1+\eta-|t|),
 &\text{ if }  T-2+\eta\le |t|< T-1+\eta,\\
 1,
 &\text{ if } |t|\ge T-1+\eta.
 \end{cases}
\end{equation}
Correspondingly, we set
\begin{equation}
 \beta_{T,\lambda}(t):=
 \begin{cases}
  R^-, &\text{ if } t\leq -\eta ,\\
   \dfrac{\eta+t}{2\eta}R^+
  + \dfrac{\eta-t}{2\eta}R^-,&
  \text{ if } |t|< \eta, \\
   R^+,  &\text{ if  } t \geq \eta,
 \end{cases}
\end{equation}
so that $\beta_{T,\lambda}'(t)=0$ for $|t|\ge \eta$, and $(\beta_{T,\lambda}, v_{T,\lambda})\in\mathcal U^T[R^-,R^+]$.
We estimate
\begin{equation}
	\begin{split}
	\int_{-T}^{T} \frac{(1-v_{T,\lambda})^2}4+|v_{T,\lambda}'|^2 \dt
	&\le \frac{2\eta}4+
	2\int_0^\infty \frac{e^{-t}}4 + \frac{e^{-t}}{4}\dt
	+ 2e^{-(T-2)} \left(\frac14+1\right)\\
&	\le 1+3e^{-T/2}+\frac\eta2
	\le g^*_\infty(R^-,R^+)+3e^{-T/2}+\frac\eta2,
	\end{split}
\end{equation}
$ f(v_{T,\lambda})|\beta_{T,\lambda}'|=0$ almost everywhere,
\begin{equation}
 \int_{-T}^T
 \lambda\,  \dist^2(\beta_{T,\lambda},\On)
  \dt
 \le
 2\eta \, \lambda \, 4d,
\end{equation}
and 
\begin{equation}
 \|\beta_{T,\lambda}'\|_{L^2((-T,T))}^2=\frac{|R^+-R^-|^2}{(2\eta)^2}2\eta\le \frac{d}{\eta}.
\end{equation}
At this point we choose $\eta:=1/(\lambda T)$ and conclude the proof.
\end{proof}

\section{Proof of  $\Gamma$-convergence}\label{sec:proofmain}

In this last section we show the $\Gamma$-convergence result by proving separately the lower bound for $\tF_\e$ and the upper bound for $\F_\e$  (recall \eqref{0904251306} and \eqref{1304251245}).

\subsection{Lower bound}\label{subsec:liminf}

This section is devoted to the proof of the following lower bound result (see (LB)' in Theorem \ref{thm:maineq}).

\begin{theorem}\label{th-lb}
	For any $(u_\varepsilon, v_\varepsilon)_\varepsilon
	\subset L^1(\Omega;\Mno/\G\times [0,1])$ and $u \in \PC(\Omega; \On/\G)$ satisfying
\[
	u_\varepsilon \to u \text{ in }L^1(\Omega; \Mno/\G),\qquad v_\varepsilon \to 1 \text{ in }L^1(\Omega).
\]
	it holds
	\begin{equation*}
		\tF^\lambda(\uu,1)\leq \liminf_{\varepsilon\to 0} \tF_\varepsilon(\uu_\e, v_\varepsilon).
	\end{equation*} 
\end{theorem}
The proof uses first a blow up argument to reduce to the case of a flat interface with the limiting $\uu$ taking only two values and then a slicing argument for functions in $H^1(\Omega;\Mno/\G)$ to reduce to a one dimensional problem in the interval
$I:=\left(-\frac 12, \frac 12 \right)$. Being in dimension one allows us to exploit the relation between $H^1(I;\Mno/\G)$ and $H^1(I;\Mno)$ discussed in Lemma \ref{l:clean}. The main additional difficulty in the treatment of  these one dimensional sections, which is handled  in a sequence of lemmas before the main proof, is the fact that the functionals $\tF_\e$ are not invariant under scaling, because the truncation with $M_\e$ of the prefactor $f_\e$ scales differently from the rest. Indeed the surface energy $g_\lambda$ defined in Section \ref{sec:surfen} does not contain any truncation.

To deal with this, we introduce an intermediate step with a further truncation in the prefactor
and
for every $j\in \N$ we replace $f$ by
\begin{equation}\label{0203241030}
h_j(s):=\frac{js}{1-s} \wedge f(s), \qquad s\in [0,1),
\end{equation}
 and we introduce the  corresponding surface energy densities
\begin{equation}\label{gbarlaj}
\begin{split}
{g}_{j,\lambda}( R^-, R^+) :=
\inf \bigg\{  \int_{-\frac 12}^{\frac 12} & \sqrt{(1-v)^2 + 4\lambda \, \dist^2(\beta, \On)} \sqrt{h_j^2(v) |\beta'|^2+|v'|^2} \dt\colon 
\\
&
(\beta,v) \in \calU^{1/2}[ R^-,R^+]  \bigg\}
\end{split}
\end{equation}
for $\lambda \in [0,+\infty)$, and 
\begin{equation}\label{gbarinfj}
	{g}_{j,\infty}( R^-, R^+) := \sup_{\lambda >0} {g}_{j,\lambda}( R^-, R^+).
\end{equation} 
The set of admissible pairs $ \calU^{1/2}[ R^-,R^+]$ is defined in \eqref{eqdefcalUT}. 
These are good approximations of $g_\lambda$ as shown in Lemma~\ref{p:approxg} below. The reason for doing this is that it permits to replace the phase field $v_\e$ with a truncated version $\tv_\e$ with a small error and, at the same time, to remove the truncation from the prefactor.

\begin{lemma}\label{p:approxg}
For  any $\lambda\in[0,+\infty]$, $R^+$, $ R^- \in \On$, the following holds:
	\begin{equation}\label{f:approxg}
	g_\lambda( R^-, R^+)\leq	\lim_{j\to +\infty}{g}_{j,\lambda}( R^-,R^+).
	\end{equation}
\end{lemma}

\begin{proof}
The proof follows closely that of \cite[Proposition~7.3 (iii)]{CFI14}.
We sketch here the main steps in the case $\lambda \in [0,+\infty)$, being the case $\lambda=+\infty$ obtained by  taking the supremum over $\lambda>0$.
For fixed $ R^-,R^+ \in \On$, for any Borel set $E\subset \R$,
we consider the functionals defined in $H^1(J;\Mno) \times  H^1(J;[0,1])$ for any interval $J$ that contains $E$ by
\begin{equation}\label{eqdefGjlambda}
\begin{split}
H_{j,\lambda} (\beta,v;E) &:=  \int_{E}  \sqrt{(1-v)^2 + 4\lambda \, \dist^2(\beta, \On)} \,\sqrt{h_j^2(v) |\beta'|^2+|v'|^2} \dt,
	\\
H_{\lambda} (\beta,v;E)& :=  \int_{E}  \sqrt{(1-v)^2 + 4\lambda \, \dist^2(\beta, \On)} \,\sqrt{f^2(v) |\beta'|^2+|v'|^2} \dt,
\end{split}
\end{equation}
so that 
\[
{g}_{j,\lambda}( R^-,R^+) = \inf_{(\beta,v)\in \calU^{1/2}[ R^-,R^+]}H_{j,\lambda}(\beta,v,I) , \qquad
{g}_{\lambda}^*( R^-,R^+) = \inf_{(\beta,v)\in\calU^{1/2}[ R^-,R^+]}H_{\lambda}(\beta,v,I).
	\]
First we note that by \eqref{0203241030} and Prop.~\ref{prop:proprietag}\ref{prop:proprietag:upperbound},
$$
{g}_{j,\lambda}( R^-,R^+)\leq {g}_{j+1,\lambda}( R^-,R^+)\leq g^*_\lambda( R^-,R^+)\leq 1,
$$
and hence we have to show that
\begin{equation}
	\label{eq-j-converse-ineq}
	{g}_{\lambda}( R^-,R^+)\leq \sup_j {g}_{j,\lambda}( R^-,R^+).
\end{equation}
	For any  $j\in\N$, let $(\beta_j,v_j) \in \calU^{1/2}[ R^-,R^+]$ be such that
	\begin{equation}\label{f:approx1}
		H_{j,\lambda} (\beta_j,v_j;I) \leq {g}_{j,\lambda}( R^-,R^+) +\frac 1j.
	\end{equation}
	If $\liminf_j \min_{[0,1]}v_j=0$, then
	the same argument as in \eqref{eqivlader} easily gives
	\begin{equation*}
\label{eq-converse-easy}
\limsup_{j\to +\infty} H_{j,\lambda}(\beta_j,v_j;I)\geq 1
	\end{equation*}
which concludes the proof.

Assume now that $\liminf_j \min_{[0,1]}v_j>0$ and pick $0<\eta<\liminf_j \min_{[0,1]}v_j$. We define
$v_j^\eta:=v_j\wedge (1-\eta)$. Extending $\beta_j$ constant outside $I$ and interpolating $v_j^\eta$ linearly in order to achieve the boundary condition $1$ in the interval $(-1,1)$, and then using the invariance under reparametrization of $H_{\lambda}$, we immediately obtain that
	\begin{equation}\label{eq-estimate-reparametrization}
		g_{\lambda}(R^+,R^-)\leq
				g^*_{\lambda}(R^+,R^-)\leq
				H_{\lambda} (\beta_j,v^\eta_j;(-1,1))\leq  H_{\lambda} (\beta_j,v^\eta_j;I)+ \eta^2.
	\end{equation}
	We subdivide $I$ in three subsets:
	\begin{equation}
		\label{eq-subsets}
		\begin{split}
			&A_j:=\{t\in I:\ v_j(t)\leq 1-\eta\},\\
			&B_j:=\{t\in I:\ 1-\eta<v_j(t)\leq 1\,,\ h_j(v_j(t))=f(v_j(t))\},\\
			&C_j:=\{t\in I:\ 1-\eta<v_j(t)\leq 1\,,\ h_j(v_j(t))<f(v_j(t))\}.\\
		\end{split}
	\end{equation}
We start with $A_j$. Here $v_j=v_j^\eta$ and since for $j$ sufficiently large $h_j=f$ in $[\eta,1-\eta]$, \begin{equation}
	\label{eq-estimate-Aj}
	H_{\lambda}(\beta_j,v_j^\eta;A_j)=H_{j,\lambda}(\beta_j,v_j;A_j)\qquad \hbox{for all $j$ large enough}.
\end{equation}
In order to treat the second set we use the same argument as in \eqref{eq-approx-f} and \eqref{eq-estimate-truncation} and obtain
\begin{equation}
	\label{eq-estimate-Bj}
H_{\lambda}(\beta_j,v_j^\eta;B_j)\leq \frac{\ell+\ol \omega(\eta)}{\ell-\ol \omega(\eta)}H_{j,\lambda}(\beta_j,v_j;B_j)\qquad \hbox{for all $j$}, \qquad \ol\omega(\eta) \to 0\ \text{as}\ \eta \to 0.
\end{equation}
Finally we observe that if $t\in C_j$ then $(1-v_j(t))h_j(v_j(t))= jv_j(t)\geq j(1-\eta)$ and therefore
\begin{equation*}
(1-v_j(t))h_j(v_j(t))\geq \eta f(1-\eta)=(1-v_j^\eta)f(v_j^\eta)
\end{equation*}
for all $j\geq \frac {\eta f(1-\eta)}{1-\eta}$. Using the same argument as for $B_j$ we conclude that
\begin{equation}
	\label{eq-estimate-Cj}
H_{\lambda}(\beta_j,v_j^\eta;C_j)\leq
H_{j,\lambda}(\beta_j,v_j;C_j)\qquad \hbox{for all $j\geq \frac {\eta f(1-\eta)}{1-\eta}$}.
\end{equation}
Then \eqref{eq-j-converse-ineq} follows by  \eqref{eq-estimate-reparametrization}, \eqref{eq-estimate-Aj}, \eqref{eq-estimate-Bj} \eqref{eq-estimate-Cj}, and \eqref{f:approx1}, taking first the limit $j\to\infty$ and then $\eta\to0$.
\end{proof}

Next, we show that in the definition of the energies $g_{j,\lambda}$ one can also consider test functions such that the boundary conditions are fulfilled only asymptotically.

\begin{lemma}\label{lemmalbGjlambdaboundaryvalues}
Fix $j\in\N$, $\lambda\in[0,+\infty)$, $R^+$, $ R^- \in \On$. Consider a sequence $(\beta_k,v_k)\in H^1(I;\Mno\times[0,1])$
such that
\begin{equation}
\beta_k \toG R^-\chi_{(-\frac12,0)}+
R^+\chi_{(0,-\frac12)} 
\text{ and } v_k\to 1 
\text{ in } L^1.
\end{equation}
Then the functional $H_{j,\lambda}$ defined in \eqref{eqdefGjlambda}  obeys
\begin{equation}
\min_{G^+,G^-\in\G}
g_{j,\lambda}(G^-R^-,G^+R^+)
\le \liminf_{k\to\infty}  H_{j,\lambda}(\beta_k,v_k;I).
\end{equation}
\end{lemma}
\begin{proof}
As usual, in order to interpolate the field $\beta_k$ we need to move $v_k$ away from 1.
Possibly passing to a subsequence, we have pointwise convergence almost everywhere, therefore we can choose $z\in (\frac14,\frac12)$ and $G_k^\pm\in\G$ such that
\begin{equation}
\alpha_k:= |G_k^+R^+-\beta_k(z)|
+|G_k^-R^--\beta_k(-z)|
+|1-v_k(z)|+
|1-v_k(-z)|
\to0;
\end{equation}
for $k$ sufficiently  large we have $h_j(s)=js/(1-s)$ for all $s\in [1-\alpha_k,1)$. Therefore
on the set $\{v_k\ge 1-\alpha_k\}$, the integrand in
the functional
 $H_{j,\lambda}(v_k,\beta_k;I)$ can be rewritten as
\begin{equation}
\sqrt{1+ \frac{4\lambda}{(1-v_k)^2}\, \dist^2(\beta_k, \On)} \,\sqrt{j^2 v_k^2 |\beta_k'|^2+(1-v_k)^2|v_k'|^2}
\end{equation}
and, setting
$\widehat v_k:=v_k\wedge (1-\alpha_k)$, one obtains
\begin{equation}
H_{j,\lambda}(\beta_k,\widehat v_k;(-z,z))
 \le H_{j,\lambda}(\beta_k,v_k;(-z,z)).
\end{equation}
We extend $\widehat v_k$ to equal $1-\alpha_k$ on the rest of $[-1,1]$, to be one
outside $(-2,2)$, and to be the affine interpolation in-between.
We then extend $\beta_k$ so that $\beta_k=G_k^+R^+$
on $(1,\infty)$, $\beta_k=G_k^-R^-$ on $(-\infty,-1)$, and the affine interpolation in $(-1,-z)$ and $(z,1)$. We
observe that both functions are in $H^1$ on $(-2,2)$.
We compute
\begin{equation}
 H_{j,\lambda}(\beta_k,\widehat v_k;(z,1))
 \le (1-z) \sqrt{\alpha_k^2+4\lambda\alpha_k^2}
 \sqrt{\frac{j^2}{\alpha_k^2} 4\alpha_k^2}
 = 2 j \alpha_k \sqrt{1+4\lambda},
\end{equation}
\begin{equation}
 H_{j,\lambda}(\beta_k,\widehat v_k;(1,2))
 =\int_1^2 \alpha_k (2-t) \alpha_k\dt = \frac12 \alpha_k^2,
\end{equation}
and the same on the other side. Since $H_{j,\lambda}$ is invariant under reparametrizations we obtain
\begin{equation}
 g_{j,\lambda}(G_k^-R^-,G_k^+R^+)
 \le H_{j,\lambda} (\beta_k,\widehat v_k; (-2,2))
 \le  H_{j,\lambda}(v_k,\beta_k;I)
 +\alpha_k^2 + 4j\alpha_k\sqrt{1+4\lambda}
\end{equation}
for all $k$ sufficiently large;
 taking the limit $k\to\infty$ the proof is concluded.
\end{proof}

We conclude with a lower bound for the one-dimensional functionals obtained from the blow-up.

\begin{lemma}\label{l:lb1d}
	Let 
	$\e_k/\delta_k \to \lambda\in [0,+\infty]$,
	$M_k\to +\infty$,
	and assume that $(\beta_k,v_k)\in H^1(I;\Mno \times [0,1])$, $I=[-1/2,1/2]$, obeys
	\[
	\beta_k \to  R^-\chi_{[-1/2,0)}+R^+\chi_{[0,1/2]}\ \text{in} \ \LG(\Omega,\Mno), \quad 
	v_k \to 1 \ \text{in} \ L^1(I)
	\]
	for some $ R^-,R^+\in\On$.
	 Then
	\begin{equation*}\label{eq-lb-1d}\begin{split}
&\min_{G^+,G^-\in\G}
g_{j,\lambda}(G^-R^-,G^+R^+)\\
&\le
 \liminf_{k\to +\infty}
\int_{{I} } ( M^2_k\wedge {\e_k}f^2(v_k) )\,|\beta'_k|^2 \dt + \mathrm{PF}_{\e_k}(v_k;I) +
\frac{1}{\delta_k} \int_{{I} }\dist^2(\beta_k, \On)  \dt .
\end{split}\end{equation*}
\end{lemma}
\begin{proof}
Let us first assume $\lambda < +\infty$. For $h_j$ defined in \eqref{0203241030}, we have 
	\begin{equation}\label{f:estlb1d1}
		\int_{{I} }  ( M^2_k\wedge {\e_k}f^2(v_k) )\,|\beta'_k|^2 \dt
		\geq \int_{I} 
		( M^2_k \wedge  \e_k h_j^2(v_k))\,| \beta_k'|^2 \dt.
	\end{equation}
Choose $\gamma_k:=1-j \sqrt{\eps_k}/M_k$
and define $\widetilde v_k:=v_k\wedge (1-\gamma_k)$.
Then
$\eps_k^{1/2} h_j(\gamma_k)
\le \eps_k^{1/2} j\gamma_k /(1-\gamma_k)\le
M_k$ and
\begin{equation}\label{f:estlb1d2}
		\begin{split}
			\int_{I} &
			(M^2_k  \wedge \e_k h_j^2(v_k))\,| \beta_k'|^2 + \frac{(1-v_k)^2}{4\e_k}+\e_k | v_\varepsilon'|^2 \dt
			\\
			& \geq \int_{I}  \e_k h_j^2(\tv_k)
			\,| \beta'_k|^2 + \frac{(1-\tv_k)^2}{4\e_k} + \e_k | \tv_k'|^2  \dt
			- \int_{\{v_k >\gamma_k\}} 	\frac{(1-\gamma_k)^2}{4\e_k} \dt
			\\
			& \geq \int_{I}  \e_k h_j^2(\tv_k)
			\,| \beta'_k|^2 + \frac{(1-\tv_k)^2}{4\e_k} + \e_k | \tv_k'|^2  \dt
			- \frac{j^2}{4M_k^2} .
		\end{split}
	\end{equation}
Therefore
\begin{equation}\begin{split}\label{0202261331}
 &	\int_{{I} } ( M^2_k\wedge {\e_k}f^2(v_k) )\,|\beta'_k|^2 \dt + \AT_{\e_k}(v_k;I) +
\frac{1}{\delta_k} \int_{{I} }\dist^2(\beta_k, \On)  \dt \\
&\ge
\int_I
 \sqrt{ (1-\tv_k)^2
 + 4 \frac{\eps_k}{\delta_k} \dist^2(\beta_k, \On)}
 \sqrt{h_j^2(\tv_k)|\beta_k'|^2 + |\tv_k'|^2} \dt
-\frac{j^2}{4M_k^2}.
 \end{split}\end{equation}
The last term tends to zero for $k\to\infty$. In order to estimate the integral, we distinguish three cases.
 If $\lambda=0$,
the integral on right-hand side of this equation is larger than or equal to $H_{j,0}(\beta_k,\tv_k)$ and we conclude using Lemma~\ref{lemmalbGjlambdaboundaryvalues}.
If $\lambda\in(0,\infty)$, the same integral is larger than or equal to
$\rho_k^{1/2}H_{j,\lambda}(\beta_k,\tv_k)$, where
$\rho_k:=1\wedge {\frac{\eps_k}{\lambda\delta_k}}\to1$,
and the conclusion follows similarly.
Indeed, for any $a,b\ge0$,
\begin{equation*}
 \rho_k^{1/2}\sqrt{a+\lambda b}
 =\sqrt{\rho_k a + \rho_k\lambda b}
 \le \sqrt{ a + \frac{\eps_k}{\delta_k} b}.
\end{equation*}
It remains to treat the case $\eps_k/\delta_k\to\infty$. Fix $\lambda'>1$.
For $k$ sufficiently large, we have
$\eps_k/\delta_k\ge\lambda'$ and therefore
the integral in \eqref{0202261331} is larger or equal than $H_{j,\lambda'}(\beta_k, \tv_k)$.
The conclusion now follows by taking the supremum in $\lambda'$, in view of \eqref{gbarinfj} (and recalling again \eqref{gbarlaj}).
\end{proof}

We are ready to prove the lower bound, by blow--up and  slicing.

\begin{proof}[Proof of Theorem \ref{th-lb}]
	If $ \liminf_{\varepsilon\to 0} \tF_\varepsilon(\uu_\e, v_\e)=+\infty$ we have nothing to prove.
	On the other hand, if
$ \liminf_{\varepsilon\to 0} \tF_\varepsilon(\uu_\e, v_\e)<+\infty$
	by Theorem \ref{t:comp1} we may assume, possibly passing to a subsequence and with a small abuse of notation, that $v=1$ a.e.\ in $\Omega$, $\uu \in \PC(\Omega;\On/\G)$, and
	\[
	\liminf_{\varepsilon\to 0} \tF_\varepsilon(\uu_\e, v_\e)= \lim_{\varepsilon\to 0} \tF_\varepsilon(\uu_\e, v_\e).
	\]
	We define the nonnegative Radon measures
	\begin{equation}\label{2002242111}
		\mu_\varepsilon:=\left( f_\varepsilon^2(v_\varepsilon) \,|\nabla \uu_\e|^2 + \frac{(1-v_\varepsilon)^2}{4\varepsilon}+\varepsilon |\nabla v_\varepsilon|^2 + \frac{1}{\delta_\varepsilon} \dG^2(\uu_\varepsilon, \On/\G) \right) \mathcal{L}^N ,
	\end{equation}
	so that $\tF_\varepsilon(\uu_\e, v_\varepsilon)=\mu_\varepsilon(\Omega)$, and the uniform bound for the energies corresponds to a uniform bound for the total variation of the measures $\mu_\e$,   and hence there exists a nonnegative Radon measure $\mu$ in $\Omega$ which is the weak*--limit of $\mu_\e$ (up to extracting a further subsequence, not relabeled). By the lower semicontinuity of the  total variation with respect to the weak* convergence of measures, we have that
	\[
	\mu(\Omega)\leq \liminf_{\varepsilon\to 0} \mu_\e (\Omega)= \lim_{\varepsilon\to 0} \tF_\varepsilon(\uu_\e, v_\varepsilon),
	\]
	and therefore the result is proved once we show that
	\begin{equation}\label{f:lob}
		\mu(\Omega) \geq	\tF^\lambda(\uu,1) ,
	\end{equation}
	or, equivalently, that for $\calH^{N-1}$-a.e. $x_0\in J_u$ it holds
	\begin{equation}\label{f:lob2}
		\frac{\di\mu}{\di(\calH^{N-1}\LL J_u)}(x_0) \geq g_\lambda (R^-,R^+)
	\end{equation}
where $R^\pm\in u^\pm(x_0)$ are any representatives of the equivalence class, the choice being irrelevant by the invariance of $g_\lambda$ under the separate action of $\G$ on its two arguments.
	
	In the following, we  prove the last inequality for $x_0\in J_u$ such that the limit
	\begin{equation}\label{f:goodp1}
	\lim_{\rho \to 0^+} \frac{\mu(Q^\nu_\rho(x_0))}{\calH^{N-1}(J_u \cap  Q^\nu_\rho(x_0))}=	\frac{\di\mu}{\di(\calH^{N-1}\LL J_u)}(x_0)
	\end{equation}
	exists finite, and
	\begin{equation}\label{f:goodp2}
		\lim_{\rho \to 0^+} \frac{\calH^{N-1}(J_u \cap  Q^\nu_\rho(x_0))} {\rho^{N-1}}=1,
	\end{equation}
	where $\nu:=\nu_u(x_0)$, and $Q^\nu_\rho(x_0)=x_0+\rho Q^\nu$ is the cube centered in $x_0$ with a face orthogonal to $\nu$ and side length  $\rho$.
	
	By \cite[Theorem 2.3]{Amb90SBVX}, $J_u$ is a $\calH^{N-1}$ countably rectifiable set, and hence \eqref{f:goodp1} and \eqref{f:goodp2} are satisfied $\calH^{N-1}$-a.e in $J_u$.
	Fixed $x_0$ as above,
	by weak$^*$--convergence of measures, we have that $\lim_{\e \to 0}\mu_\e (Q^\nu_\rho(x_0))= \mu (Q^\nu_\rho(x_0))$ for every $\rho$ such that $\mu(\partial Q^\nu_\rho(x_0))=0$. Then, by diagonalization, we can find sequences
	$\rho_k\to0$,  $\e_k\to0$ 
	with $\e_k/\rho_k\to0$,
$\delta_{\e_k}/\rho_k\to0$ and
	such that, setting  $u_k:=u_{\e_k}$, $v_k:=v_{\e_k}$, it holds
	\[
	\frac{\di\mu}{\di(\calH^{N-1}\LL J_u)}(x_0)= \lim_{k\to +\infty}\frac{\mu_{\e_k}(Q^\nu_{\rho_k}(x_0))}{\rho_k^{N-1}}=
	\lim_{k\to +\infty} \frac 1{\rho_k^{N-1}}\, \tF_{\e_k} (u_k,v_k;  Q^\nu_{\rho_k}(x_0))
	\]
	and
	\begin{equation*}
\rho_k^{-N}\| \dG(u_k,[\overline R](\cdot+x_0)) + |v_k-1| \, \|_{L^1(Q^\nu_{\rho_k}(x_0))}\to0,
\end{equation*}
	where
\[
\overline R:= R^- \chi_{\{x\cdot \nu <0\} }+ R^+ \chi_{\{x\cdot \nu >0\}}.
\]

	Changing variable to $y:=(x-x_0)/\rho_k\in Q^\nu$ in the integrals defining $\tF_{\e_k}  (u_k,v_k;  Q^\nu_{\rho_k}(x_0))$, and denoting as usual for  the rescaled functions  by
	\[
	\widehat\uu_k(y):=\uu_k (x_0+\rho y), \quad \widehat v_k(y):=v_k (x_0+\rho y), \qquad y\in Q^\nu,
	\]
	we get
$ \dG(\widehat u_k,[\overline R])\to0$,
$\widehat v_k\to1$ in $L^1(Q^\nu)$, and
	\[
	\begin{split}
		\frac 1{\rho_k^{N-1}}\, & \tF_{\e_k}  (u_k,v_k;  Q^\nu_{\rho_k}(x_0)) \\ &=
		\int_{Q^\nu} (\widehat M_{k}^2\wedge \widehat\e_k f^2(\widehat v_k))
		\,|\nabla \widehat\uu_k |^2 + \frac{(1-\widehat v_k)^2}{4\widehat \eps_k} +\widehat \eps_k |\nabla \widehat v_k|^2 +
		\frac{1}{\widehat\delta_k} \dG^2(\widehat\uu_k , \On/\G) \dy
	\end{split}
	\]
	for 
	\[
	\widehat M_{k}:= \frac{M_{\e_k}}{\sqrt{\rho_k}}  \quad
	\widehat\eps_k:= \frac{\varepsilon_k}{\rho_k}, \quad  \widehat\delta_k:=\frac{\delta_{\e_k}}{\rho_k}, \quad \text{so that}\quad
	\lim_{k\to \infty} \frac{\widehat\eps_k}{\widehat\delta_k}= \lim_{k\to \infty} \frac{\eps_k}{\delta_k}=\lambda.
	\]
	Hence, if we consider the slices of these functions in the direction $\nu$, defined for   $z \in \Pi^\nu\cap  Q^\nu:=\{z\in Q^\nu \colon z\cdot \nu=0\}$ as
	\[
	(\widehat \uu_k)^\nu_z(t):= \widehat \uu_k(z+t\nu), \quad (\widehat v_k)^\nu_z(t):= \widehat v_k(z+t\nu), 
	\qquad t\in I=\left[-\frac 12, \frac 12 \right],
	\]
	then, by Proposition \ref{prop-slicing}, for $\calH^{N-1}$--a.e. $z\in \Pi^\nu\cap  Q^\nu$,
	\begin{equation}\label{f:slilb1}
		((\widehat \uu_k)^\nu_z, (\widehat v_k)^\nu_z)\in H^1(I;\Mno/\G\times [0,1]),
	\end{equation}
	and
	\begin{equation}\label{f:slilb2}
		\begin{split}
			(\widehat \uu_k)^\nu_z  & \to [R^-] \chi_{[-1/2,0)}+ [R^+] \chi_{[0,1/2]}
			\quad\text{in }L^1(I; \Mno/\G), \\
			(\widehat v_k)^\nu_z &  \to 1 \quad\text{in } L^1(I).
		\end{split}
	\end{equation}
	By Lemma \ref{l:clean}, for every $k\in\N$, and for every $z\in \Pi^\nu\cap  Q^\nu$ such that \eqref{f:slilb1} holds true,
	there exists $\beta_{k,z}\in H^1(I;\Mno)$ such that $[\beta_{k,z}]=(\widehat \uu_k)^\nu_z $ and 
	$|\beta_{k,z}'|=|((\widehat \uu_k)^\nu_z)'|$, and hence 
	\[
	\begin{split}
		\widehat F_k((\widehat \uu_k)^\nu_z, & (\widehat v_k)^\nu_z;I)
		\\  :=
		\int_{I} & (\widehat M_{k}^2\wedge \widehat\e_k f^2((\widehat v_k)^\nu_z))
		\,|((\widehat \uu_k)^\nu_z)' |^2 + \frac{(1-(\widehat v_k)^\nu_z)^2}{4\widehat \eps_k} +
		\widehat \eps_k |((\widehat v_k)^\nu_z)'|^2 +
		\frac{1}{\widehat\delta_k} \dG^2((\widehat \uu_k)^\nu_z , \On/\G) \dt  \\
		=
		\int_{I} & (\widehat M_{k}^2\wedge \widehat\e_k f^2((\widehat v_k)^\nu_z))
		\,|\beta_{k,z}' |^2 + \frac{(1-(\widehat v_k)^\nu_z)^2}{4\widehat \eps_k} +
		\widehat \eps_k |((\widehat v_k)^\nu_z)'|^2 +
		\frac{1}{\widehat\delta_k} \dist^2(\beta_{k,z} , \On) \dt.
	\end{split}
	\]
	Moreover, by \eqref{f:slilb2},
	\begin{equation*}
		 \beta_{k,z} \toG R^-\chi_{[-1/2,0)}+R^+\chi_{[0,1/2]},
		\qquad \text{for}\ \calH^{N-1}-\text{a.e.}\ z\in \Pi^\nu\cap Q^\nu.
	\end{equation*}
By Lemma~\ref{p:approxg} and Lemma \ref{l:lb1d},
we obtain
	\[
	\liminf_{k\to +\infty} \widehat F_k((\widehat \uu_k)^\nu_z,  (\widehat v_k)^\nu_z;I)
	\geq g_\lambda (R^-,R^+),
	\]
	for $\calH^{N-1}$--a.e. $z\in \Pi^\nu\cap Q^\nu$.
	
	The validity of \eqref{f:lob2}, and hence the conclusion, now follows by integrating on $\Pi^\nu$, and using Fubini's theorem, the estimates in Proposition \ref{prop-slicing}, and Fatou's lemma:
	\[
	\begin{split}
		\frac{\di\mu}{\di(\calH^{N-1}\LL J_u)}(x_0) &
		=  \lim_{k\to +\infty} \frac 1{\rho_k^{N-1}}\, \tF_{\e_k} (u_k,v_k;  Q^\nu_{\rho_k}(x_0)) 
		\geq   \lim_{k\to +\infty} 
		\int_{\Pi^\nu\cap Q^\nu} \widehat F_k((\widehat \uu_k)^\nu_z,  (\widehat v_k)^\nu_z;I)\,  \di z
		\\  &
		\geq \int_{\Pi^\nu\cap Q^\nu} \liminf_{k\to +\infty} \widehat F_k((\widehat \uu_k)^\nu_z,  (\widehat v_k)^\nu_z;I)\, \di z
		\geq g_\lambda (R^-,R^+)
		. \qedhere
	\end{split}
	\]
\end{proof}

\begin{remark}\label{rem:1404251131}
	We have proven above (LB)' in Theorem~\ref{thm:maineq}.  The analogous (LB) in Theorem~\ref{thm:main} can be  obtained either directly, following the same proof proposed in Theorem~\ref{th-lb}, or as a consequence of  Remark~\ref{rem:1104251642}, \eqref{0904251306}, \eqref{1304251245}, and Theorem~\ref{th-lb}.
\end{remark}

\subsection{Upper bound}\label{Subsec:limsup}

In this section we prove the upper bound for the functionals $\F_\e$ stated in Theorem \ref{thm:main}(UB).

\begin{theorem}\label{th-upperbound}
	For any $\beta \in \PC(\Omega; \On)$ there exists $(\beta_\varepsilon, v_\varepsilon)_\varepsilon \subset \SBVG(\Omega;\Mno )\times H^1(\Omega)$, 
with 
\[
	\beta_\varepsilon \to \beta \text{ in }\LG(\Omega; \Mno),\qquad v_\varepsilon \to 1 \text{ in }L^1(\Omega),
\]
and  such that
	\begin{equation}\label{limsup}
		\F^\lambda(\beta,1)\geq \limsup_{\varepsilon \to 0} \F_\varepsilon(\beta_\varepsilon, v_\varepsilon).
	\end{equation}
Moreover, for any sequence $\eta_\eps\to0$ such that
		\begin{equation}\label{eqassetaeps}
			\lim_{\eps\to0}  \frac{\eta_\eps }\eps=0\text{ and }
\lim_{\eps\to0}  \frac{\eta_\eps}{\delta_\eps}=0
		\end{equation}
		one can choose the sequences $(\beta_\e,v_\e)$  such that $\eta_\eps\|\nabla \beta_\eps\|_{L^2(\Omega)}^2\to0$.
\end{theorem}
If $\lambda<\infty$ then it suffices to consider the first condition in \eqref{eqassetaeps}, if $\lambda=\infty$ it suffices to consider the second one.

\begin{proof}[Proof of Theorem~\ref{th-upperbound}]
	We first remark that the density result
	\cite[Lemma~4.1]{BCG14}, valid for sharp-interface energies depending on the jump,  holds also for our energies 
$\F^\lambda$ that depend on the two traces. Namely, with the very same proof of  \cite[Lemma~4.1]{BCG14} one obtains that for every
	given $\beta \in \SBV_0(\Omega; \On)$ there exist $\beta_k \in \SBV_0(\Omega_k; \On)$ such that
	\begin{equation*}
\Omega\subset\subset\Omega_k \qquad
\beta_k \to \beta \quad \text{ in }L^1(\Omega; \On), \qquad
\F^\lambda(\beta,1)=\lim_{k\to +\infty}\F^\lambda(\beta_k,1; \Omega_k),
	\end{equation*}
	\begin{equation*}
J_{\beta_k} \text{ is a finite union simplexes of dimension $N-1$.}
	\end{equation*}
The map $\beta_k$
 constructed in  \cite[Lemma~4.1]{BCG14}
takes values in $\On$ because  it coincides in any point of $\Omega$ with a suitable value of $\beta$.

	Therefore we can assume that
	$\Omega\subset\subset\Omega'$,
	$\beta\in \SBV_0(\Omega';\On)$ is piecewise constant on finitely many convex polyhedra,
	whose faces are unions of simplexes of dimension $N-1$, which are denoted by $F_l$.
	We need to find sequences $(\beta_\eps,v_\eps)\to (\beta,1)$ in $L^1(\Omega)$ such that
	\begin{equation*}
\limsup_{\eps\to0} \F_\eps(\beta_\eps,v_\eps)\le \F^\lambda(\beta,1;\Omega')
=\sum_l \mathcal H^{N-1}(F_l) g_\lambda(R^-_l, R^+_l),
	\end{equation*}
	where $R^\pm_l\in\On$ are the two traces of $\beta$ on the face $F_l$.
	We start with a decomposition of the domain, referring to Figure~\ref{fig-upper1} for an illustration.
 	We denote by $\nu_l$ the normal to $F_l$
 	and  by $\partial'F_l$
	the relative boundary of $F_l$, in the sense of the boundary of $F_l$ seen as a subset of a hyperplane.
	Possibly subdividing some of the facets, we can ensure that
	they have disjoint $(N-1)$-dimensional interior, in the sense that
	the sets $F_l\setminus\partial'F_l$ are disjoint.
	For $y\in F_l$ we write $d_l(y)$ for the distance to the boundary, $d_l(y):=\dist(y,\partial'F_l)$, and write briefly $\Gamma:=\bigcup_l \partial'F_l$.
	
	In what follows we denote by $(E)_r$ the $r$-neighborhood of a set $E\subset \R^N$.
	
	Since each $\partial' F_l$ is a finite union of $(N-2)$-dimensional polyhedra, we have
	(see, e.g., \cite[Sect. 2.13]{AmbFusPal00})
	\begin{equation}
		\label{eqminkowskilimit}
		\lim_{r\to0} \frac{|(\partial' F_l)_r|}{\pi r^2}=\calH^{N-2}(\partial' F_l)<\infty,
	\end{equation}
	so that there is $C$ with
	\begin{equation}\label{eqminkowskiboundary}
		|( \Gamma)_r|\le C r^2 \text{ for all } r\in (0,1).
	\end{equation}
Fix $L\ge 1$, we shall in the end take $\eps\to0$, $L\to\infty$, with $L\eps\to0$ and other conditions that will become clear later. We observe that,
since the $F_l$ are $(N-1)$-dimensional  simplexes with disjoint interior, there is $c_*\in[1,\infty)$ such that the sets
\begin{equation}\label{eqefEleps}
E^l_\eps:= \{ y+t\nu_l: y\in F_l,
	d_l(y)>c_* L\eps,\ t\in (-L\eps, L\eps)\}
\end{equation}
are pairwise disjoint if $L\eps$ is sufficiently small, see Figure~\ref{fig-upper1}.
In these regions we can use the optimal profiles; we shall than need to interpolate around $\Gamma$.
	\begin{figure}
		\centerline{\includegraphics[width=14cm]{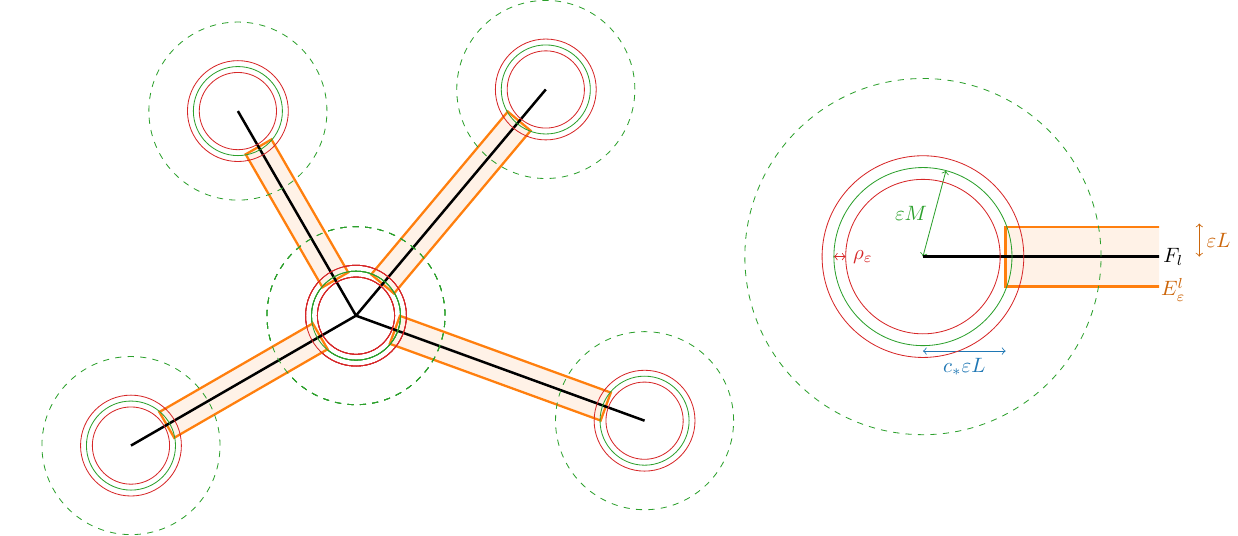}}
		\caption{Sketch of the construction in the upper bound. The black lines represent the facets $F_l$, and the orange subset is  $E^l_\eps$. 
		The dashed green 
		circles with radius $2M\eps$ around $\partial'F_l$ represent the set
			$\{v^0_\eps=0\}$, the two red ones 
			with radius $M\eps\pm \rho_\eps$ denote the area  where the two interpolations in 
			\eqref{eq-beta-finale} take place.
			The right panel shows a blow-up of a part of the set with some labels.}
		\label{fig-upper1}
	\end{figure}

We fix $M:=(1+c_*)L$, so that the lateral boundaries of each $E^l_\eps$ are contained in
$(\partial'F_l)_{M\eps}$, in the sense that
\begin{equation}
 \partial E^l_\eps \subset (F_l+L\eps\nu_l)\cup (F_l-L\eps\nu_l) \cup (\partial'F_l)_{M\eps}
\end{equation}
(the relevance of this condition will become clear after \eqref{eqdefvepsl} below)
and define $v_\eps^0\in H^1(\Omega;[0,1])$ by
	\begin{equation*}
	v^0_\eps(x):=1\wedge \frac1{M\eps} \dist(x,(\Gamma)_{2M\eps}).
	\end{equation*}
	We observe that $\|\nabla v^0_\eps\|_{L^\infty}\le 1/(M\eps)$
	and that $v^0_\eps=1$ outside
	$(\Gamma)_{3M\eps}$, so that
	\eqref{eqminkowskiboundary} implies
	\begin{equation}\label{eqlimsupATeps}
		\AT_\eps(v^0_\eps)
		\le |(\Gamma)_{3M\eps}| \left(\frac1\eps + \frac{\eps}{(M\eps)^2}\right)\leq C L^2\eps.
	\end{equation}
	For each $l$ we pick
	$G_l\in\G$ such that $g_\lambda(R_l^-,R_l^+)= g_\lambda^*(R_l^-, G_l^{-1}R_l^+)$
	(see 
\eqref{glambdacar} and
\eqref{glambdainfty}). 
If $\lambda<\infty$, by \eqref{glambda} we can
select
	$(\widetilde \beta_l,\widetilde v_l)\in H^1((-L,L);\Mno\times[0,1])$
	such that
	\begin{equation}\label{eqchoidetildevbeta}
\begin{split}
	\widetilde v_l(\pm L)=1, \hskip2mm
\widetilde\beta_l(-L)= R^-_l,\hskip2mm
G_l\widetilde\beta_l(L)=R^+_l,\hskip2mm
\|\widetilde\beta_l\|_{L^\infty}\le \sqrt d,\\
\|\widetilde \beta_l'\|_{L^2((-L,L))}^2 \le C(L)\,,\hskip1cm
\F_\lambda^L(\widetilde\beta_l,\widetilde v_l)\le
g_\lambda(R^+_l, R^-_l)+\omega(L),
\end{split}\end{equation}
with $\omega(L)\to0$ as $L\to\infty$, and $C(L)$ a positive constant depending only on $L$.
If instead $\lambda=\infty$, we use
Lemma~\ref{lemmalambdainfty}\ref{lemmalambdainftyub} and for any fixed $\eps>0$ and $L>1$ we take
\begin{equation}\label{eq-lambdaepsL}
	\lambda_{\eps}:=\frac{\eps}{\delta_\eps}
\end{equation}
and select
	$(\widetilde \beta_l,\widetilde v_l)\in H^1((-L,L);\Mno\times[0,1])$, depending on $\eps$ and $L$,
so that the first line of \eqref{eqchoidetildevbeta} still holds, and the second one is replaced by
\begin{equation}\label{eqchoidetildevbetainfty}
\|\widetilde \beta_l'\|_{L^2((-L,L))}^2 \le C \lambda_\eps L\,,\hskip1cm
 \F_{\lambda_{\eps}}^L(\widetilde\beta_l,\widetilde v_l)\le
g_\infty(R^+_l, R^-_l) +\omega({L}) + \omega({\lambda_{\eps}}),
\end{equation}
with $\omega(s)\to 0$ as $s\to +\infty$. 
	We define $v_\eps^l:E^l_\eps \to[0,1]$, using $y$ and $t$ as in \eqref{eqefEleps}, by
	\begin{equation}\label{eqdefvepsl}
		v_\eps^l(y+t\nu_l):=
		\widetilde v_l(t/\eps ),
	\end{equation}
	and then $v_\eps$ by
	\begin{equation}\label{eqdefveps}
		v_\eps(x):=\begin{cases}
			v_\eps^0(x)\wedge
			v_\eps^l(x), & \text{ if } x\in E^l_\eps,\\
			v_\eps^0(x), & \text{ if } x\in \Omega\setminus \bigcup_l E^l_\eps\,.
		\end{cases}
	\end{equation}
	We next show that no jumps are introduced on the boundaries of any $E_l^\eps$. Indeed,
	let $x=y+t\nu_l\in\partial E_l^\eps$.
	If $d_l(y)> c_* L\eps$ then
	$|t|= L\eps$, which implies $v_\eps^l(x)=\widetilde v_l(t/\eps)=1$. If instead
	$d_l(y)<c_*L\eps$ then
	$\dist(x,\partial' F_l)\le L(1+c_*)\eps=M\eps$, so that
	$v^0_\eps(x)=0$. In both cases, the minimum in the first line of \eqref{eqdefveps} equals $v^0_\eps(x)$, and we obtain $v_\eps\in H^1(\Omega;[0,1])$.
	The monotonicity condition
	$\AT_\eps(v'\wedge v'';A)\le
	\AT_\eps(v';A)+
	\AT_\eps(v'';A)$
	leads directly to
	\begin{equation}\label{eqdecompATub}
		\AT_\eps(v_\eps)
		\le \AT_\eps(v^0_\eps)
		+ \sum_l
		\AT_\eps(v_\eps^l;E^l_\eps).
	\end{equation}
	
	We now turn to the construction of $\beta_\eps$, which is more subtle.
	To make the strategy clear we deal with each difficulty separately, defining a sequence of intermediate functions.

	In order to rescale and account for the matrices $G_l$, we replace the functions $\widetilde\beta_l$ identified in \eqref{eqchoidetildevbeta} by
	\begin{equation}\label{eqconstrbeta1}
		 \beta^l_\eps(y+t\nu_l):=
		\begin{cases}
			G_l\widetilde \beta_l(\frac{t}{\eps}), &\text{ if } t\in[0,L\eps],\\
			\widetilde \beta_l(\frac{t}{\eps}), &\text{ if } t\in [-L\eps,0).
		\end{cases}
	\end{equation}
	We can  easily use this profile in the inner regions $E^l_\eps$. It automatically matches $R_l^\pm$ for $t =\pm  L\eps$, it can jump only in 0,
	and since $\widetilde \beta_l$ is continuous the traces differ by a factor $G_l$.
	The critical point is the interpolation at the lateral boundaries which are contained in $(\partial' F_l)_{M\eps}$.
	The condition on the jump required by $\SBVG$  is not linear (in the sense that the sum of two $\SBVG$ functions is not necessarily in $\SBVG$), therefore we cannot interpolate.
	However, $\SBVG$ is one-homogeneous, in the sense that
	if $f\in \SBVG(\Omega;\Mno)$ and $\psi\in C^1(\overline \Omega)$ then
	$\psi f\in \SBVG(\Omega;\Mno)$. On the other hand, the
	term $\delta_\eps^{-1}\dist^2(\cdot,\On)$ in the energy may be significantly increased by this operation, therefore we must make sure that $\{\psi\ne1\}$ is small.
		
	Since $v_\eps^0=0$ in $(\Gamma)_{2M\eps}$, we modify the functions $\beta^l_\eps$ by interpolating first with $0$ and then with the identity.
	We fix $\varphi\in C^\infty(\R;[0,1])$, such that $\varphi(t)=0$ for $t<0$ and $\varphi(t)=1$ for $t\geq1$, let $\rho_\eps\in(0,\eps]$ be a parameter to be chosen below, and set
	\begin{equation}
		\label{eq-cut-off}
		\psi^{\out}_\eps(x):=\varphi\left(\frac{\dist(x,\Gamma)-M\eps}{\rho_\eps}\right)\qquad \psi^\iin_\eps(x):=\varphi\left(\frac{M\eps-\dist(x,\Gamma)}{\rho_\eps}\right),
	\end{equation}
We then define
\begin{equation}\label{eqdefhatbetaeps}
		\widehat \beta_\eps(x):=
		\begin{cases} \beta^l_\eps(x), & \text{ if } x\in E^l_\eps,\\
			\beta(x),  & \text{ if } x\in \Omega\setminus \bigcup_lE^l_\eps,
		\end{cases}
	\end{equation}
	and combining it with the functions above 
	\begin{equation}
\label{eq-beta-finale}
\beta_\eps(x):=\psi^\out_\eps(x)	\widehat \beta_\eps(x) +\psi_\eps^\iin(x) \Id,
	\end{equation}
which may jump only on $\bigcup_l F_l$, with
$\beta_\eps^\pm=\psi^\out_\eps \widehat \beta_\eps^\pm$.

	We have already shown that $v_\eps\in H^1(\Omega;[0,1])$; from the energy estimate it will follow $v_\eps\to1$ in $L^2(\Omega)$. From the construction we also have $\beta_\eps\in \SBV(\Omega;\Mno)$ with $J_{\beta_\eps}\subset \bigcup_l F_l$
	and $\beta_\eps^+=G_l\beta_\eps^-$ on $F_l$; in order to prove $\beta_\eps\in \SBVG(\Omega;\Mno)$ it only remains to show that $\nabla\beta_\eps\in L^2(\Omega)$, this will be done in  \eqref{eqL2estimate} below. Finally, we notice that, since $\rho_\eps\le M\eps$, in
	$\left[\Omega\setminus \bigcup_l E^l_\eps \right]\setminus (\Gamma)_{M\eps+\rho_\eps}$ we have
	$\beta_\eps=\beta$, and therefore $\nabla \beta_\eps=0$.  As
	$\|\beta_\eps\|_{L^\infty}\le\sqrt d$,  we immediately obtain $\beta_\eps\to\beta$ in $L^1(\Omega;\Mno)$.

	We now estimate the different terms of the energy.
	In $(\Gamma)_{2M\eps}$ we have
	$v_\eps^0=0$ and therefore
	$f_\eps(v_\eps)=0$, so that the term
	$f^2_\eps(v_\eps)|\nabla\beta_\eps|^2$ vanishes outside $\{\psi^\out=1\}\cap\bigcup_l E_l^\eps$.
	Using
	\eqref{eqdecompATub} leads to
	\begin{equation}\label{equbdecdomfct}
		\F_\eps(\beta_\eps, v_\eps)
		\le \sum_l
		\F_\eps(\beta_\eps^l,v^l_\eps;E_l^\eps)+\AT_\eps(v_\eps^0)
		+\frac1{\delta_\eps} \int_{(\Gamma)_{M\eps+\rho_\eps}}
		\dist^2(\beta_\eps, \On)\dx.
	\end{equation}
	We next estimate these three terms,  starting with $\F_\eps(\beta_\eps^l,v^l_\eps;E^l_\eps)$, for a fixed $l$.

 We denote by $\widehat\lambda$ either $\lambda<+\infty$ or $\lambda_\eps$ in the case $\lambda=+\infty$.
Recalling that
	$|\beta^l_\eps|\le\sqrt d$ pointwise, we have
	\begin{equation}\label{eqestdeltalambda}
		\frac{1}{\delta_\eps} \dist^2(\beta_\eps^l(y+t\nu_l),\On)
		\le
		\frac{\widehat\lambda}{\eps} \dist^2(\beta_\eps^l(y+t\nu_l),\On)
		+ \frac{2d}{\eps}\left(\frac{\eps}{\delta_\eps}-\widehat\lambda\right)_+ .
	\end{equation}
Using \eqref{eqestdeltalambda}, Fubini's theorem, and changing variables to $s:=t/\eps$,  we obtain
	\begin{equation}\label{eqFubiniRl}
		\begin{split}
			\F_\eps(\beta_\eps^l,v^l_\eps;E^l_\eps)
			\le&\int_{\{y\in F_l: d_l(y)> c_*L\eps\}}
			\int_{-L}^{L}
\Bigl[			\frac{f_\eps^2(\widetilde v_l)}{\eps}
			|\widetilde \beta'_l|^2  + \widehat\lambda\,\dist^2(\widetilde \beta_l,\On) \\
			& \hskip3cm+
			2d\left(\frac{\eps}{\delta_\eps}-\widehat\lambda\right)_++
			\frac{(1-\widetilde v_l)^2}{4}
			+ |\widetilde v'_l|^2 \Bigr]\ds\dhN(y).
	\end{split}\end{equation}
	Estimating  $f_\eps^2(\widetilde v_l)\le {\eps}f^2(\widetilde v_l)$,
	and then using \eqref{eqchoidetildevbeta} and \eqref{eqchoidetildevbetainfty} (in the case $\lambda=+\infty$), leads to
	\begin{equation}\label{eqestFepsRepsl}
		\begin{split}
			\F_\eps(\beta_\eps^l,v^l_\eps;E^l_\eps)
			\le& \calH^{N-1}(F_l)
			\left[g_\lambda(R^-_l,R^+_l)+4dL \left(\frac{\eps}{\delta_\eps}-\widehat\lambda\right)_++\omega(L)+\omega(\lambda_\eps)\right],
	\end{split}\end{equation}
with the term $\omega(\lambda_\eps)$ only present if $\lambda=\infty$.
	For later reference, we note that a similar but simpler computation, using
	again \eqref{eqchoidetildevbeta} and \eqref{eqchoidetildevbetainfty},
leads to
\begin{equation}\label{eqestimatenablabetaepsRl}
\begin{split}
\|\nabla \beta_\eps\|_{L^2(E^l_\eps)}^2
\le \frac1{\eps}\calH^{N-1}(F_l) \|\widetilde \beta'_l\|_{L^2(\R)}^2
\le C(L) \frac{1+\widehat \lambda}
{\eps}
.
\end{split}\end{equation}
	The second term in \eqref{equbdecdomfct} was already treated in \eqref{eqlimsupATeps}.
	In the third term we use the bound $\|\beta_\eps\|_{L^\infty}\le\sqrt d$ 
	and the fact that $\Gamma$ is a finite union of $(N-2)$-dimensional polyhedra
	to obtain
	\begin{equation}\label{equbdecdomfct3}
\frac1{\delta_\eps} \int_{(\Gamma)_{M\eps+\rho_\eps}}
\dist^2(\beta_\eps, \On)\dx
\le \frac{C}{\delta_\eps}| (\Gamma)_{M\eps+\rho_\eps}\setminus (\Gamma)_{M\eps-\rho_\eps}|
\le C \frac{\rho_\eps L\eps }{\delta_\eps}.
	\end{equation}
	We are now ready to conclude the estimate of the energy. 
	
	We first consider the case $\lambda<+\infty$, taking the limit as $\eps\to 0$ first with $L$ fixed.
By \eqref{equbdecdomfct3}
		if we  choose $\rho_\eps$ such that
		\begin{equation}\label{eqassrhoeps}
				\lim_{\eps\to0} \frac{\rho_\eps \eps}{\delta_\eps}=0,
			\end{equation}
		then
		\begin{equation}\label{eqestdistcore}
				\limsup_{\eps\to0}\frac1{\delta_\eps} \int_{\Omega\setminus \bigcup_l E^l_\eps}
				\dist^2(\beta_\eps, \On)\dx=0.
			\end{equation}
	Starting from \eqref{equbdecdomfct}, taking the limit as $\eps\to0$, using
	\eqref{eqestFepsRepsl}, \eqref{eqlimsupATeps}, and
	\eqref{eqestdistcore} leads to
	\begin{equation*}
		\limsup_{\eps\to0}
		\F_\eps(\beta_\eps, v_\eps)
		\le \sum_l \calH^{N-1}(F_l)
		\left( g_\lambda(R^-_l,R^+_l)+ \omega(L)\right),
	\end{equation*}
	so that 
	\begin{equation*}
		\limsup_{L\to\infty}
		\limsup_{\eps\to0}
		\F_\eps(\beta_\eps, v_\eps)
		\le \sum_l \calH^{N-1}(F_l)
		g_\lambda(R^-_l,R^+_l)=\F^\lambda(\beta,1),
	\end{equation*}
	and the conclusion follows from a diagonal argument.
	
	It remains to address
	the $L^2$ estimate on $\nabla \beta_\eps$.
	We recall the definitions in \eqref{eq-beta-finale} and \eqref{eq-cut-off}, as well as
	the bound $|\widehat\beta_\eps|\le\sqrt d$, and write the pointwise estimate
	\begin{equation}\label{eqpwboundnablabetaeps}
		|\nabla\beta_\eps|\le \sqrt d (|\nabla \psi^\out_\eps|+|\nabla\psi^\iin_\eps|)
		+ |\psi^\out_\eps \nabla \widehat \beta_\eps|.
	\end{equation}
	We treat the two terms separately. First,
	\begin{equation}\label{eqnablabetal2a}
		\|\nabla \psi^\out_\eps\|_{L^2(\Omega)}^2+	\|\nabla \psi^\iin_\eps\|_{L^2(\Omega)}^2
		\le \frac{C}{\rho_\eps^2}
		|(\Gamma)_{M\eps+\rho_\eps}\setminus(\Gamma)_{M\eps-\rho_\eps}|\le C\frac{M\eps}{\rho_\eps},
	\end{equation}
	where we used that $\Gamma$ is a finite union of $N-2$ dimensional polyhedra.
	For the second one, from \eqref{eqdefhatbetaeps} and $\nabla\beta=0$ we see that we can restrict to $\bigcup_l E^l_\eps$, that were already treated.
	Recalling
	\eqref{eqpwboundnablabetaeps},
	\eqref{eqestimatenablabetaepsRl}, and
	\eqref{eqnablabetal2a},
	\begin{equation}\label{eqL2estimate}
		\|\nabla\beta_\eps\|_{L^2(\Omega)}^2\le
		C\left( \frac{L\eps}{\rho_\eps}+C(L)\frac{1+\widehat\lambda}{\eps}\right).
	\end{equation}
	
	It remains to choose $\rho_\eps$.
	We recall \eqref{eqassrhoeps} and see that we need to choose $\rho_\eps\to0$, $0<\rho_\eps\le\eps$, such that
	\begin{equation}\label{eqconditionsrhoeps}
		\lim_{\eps\to0}\frac{\rho_\eps \eps}{\delta_\eps}=0
		\text{ and }
		\lim_{\eps\to0}\eta_\eps \left[ \frac1\eps+\frac{ \eps}{\rho_\eps}\right]=0.
	\end{equation}
	We fix $\rho_\eps:=\eps$ and see that it obeys all requirements provided that
	\begin{equation}\label{eq-choiceeta}
		\lim_{\eps\to0}  \frac{\eta_\eps }\eps=0,
\end{equation}
as in the assumptions.

With this choice of $\eta_\eps$ we then obtain that
	\begin{equation}\label{eqL2estimateeta}
\limsup_{L\to +\infty}\limsup_{\eps\to 0}\eta_\eps	\|\nabla\beta_\eps\|_{L^2(\Omega)}^2=0
\end{equation}
so that with a diagonal argument we conclude the proof in the case $\lambda<+\infty$.

We now turn to the case $\lambda=+\infty$.  We fix a sequence $L_\eps\to +\infty$ chosen below, that will depend on the sequence $\lambda_\eps$ defined in \eqref{eq-lambdaepsL}.

Now starting from  \eqref{equbdecdomfct},   collecting
\eqref{eqestFepsRepsl}, \eqref{eqlimsupATeps}, \eqref{equbdecdomfct3} leads to 
\begin{equation}\label{lastupperlambdainfty}
	\F_\eps(\beta_\eps, v_\eps)
	\le \sum_l  \calH^{N-1}(F_l)
 	\left[g_\infty(R^-_l,R^+_l)+\omega_\eps\right] 
 	 	+C L_\eps^2\eps + C\, \frac{\rho_\eps L_\eps\eps}{\delta_\eps}
\end{equation}
where $\omega_\eps=\omega(L_\eps)+\omega(\lambda_\eps)$.
Then choosing $L_\eps\to\infty$ so that $L_\eps^2\eps\to 0$ as $\eps\to 0$, and recalling \eqref{eqassrhoeps}, we conclude that
$$
\limsup_{\eps\to0}
	\F_\eps(\beta_\eps, v_\eps)\leq \sum_l \calH^{N-1}(F_l)
	g_\infty(R^-_l,R^+_l)=\F^\infty(\beta,1).
$$
In order to obtain the estimate of the gradients, in view of \eqref{eqL2estimate} it is enough to replace the conditions \eqref{eqconditionsrhoeps} 
with
	\begin{equation}\label{eqconditionsrhoepsinfty}
	\lim_{\eps\to0}\frac{\rho_\eps L_\eps \eps}{\delta_\eps}=0
	\text{ and }
	\lim_{\eps\to0}\eta_\eps \left[ C(L_\eps)\frac{\lambda_\eps}\eps+\frac{ L_\eps\eps}{\rho_\eps}\right]=0.
\end{equation}
We fix $\rho_\eps:=\delta_\eps\wedge\eps$, so that the first condition in \eqref{eqconditionsrhoepsinfty} holds. We can choose $L_\eps\to\infty$ such that the second one holds if and only if 
\begin{equation}
 \lim_{\eps\to0}
\eta_\eps \frac{\lambda_\eps}{\eps}
= \lim_{\eps\to0} \frac{\eta_\eps}{\delta_\eps}=0,\text{ and }
 \lim_{\eps\to0} \eta_\eps\frac{\eps}{\rho_\eps}=0.
 \end{equation}

\end{proof}

\begin{remark}\label{rem:1404251131}
The upper bound (UB)' in Theorem~\ref{thm:maineq} follows directly from Theorem~\ref{th-upperbound} and Theorem~\ref{theoliftingmanifold}. Namely, given $u \in \SBV_0(\Omega;\On/\G)$,
by Remark~\ref{rem:1104251642}, \eqref{0904251306}, and \eqref{1304251245}, a recovery sequence for $\widetilde\F^\lambda(u,1)$ is given by
$([\beta_\e],v_\e)_\e$, for $(\beta_\e,v_e)_\e$ recovery sequence for $\F^\lambda(\beta,1)$, 
$\beta \in \SBV_0(\Omega;\On)$ being the lifting of $u$ given in Theorem~\ref{theoliftingmanifold} 
(see also Remark~\ref{rem:0904251748}).

Therefore the proofs of Theorems~\ref{thm:main} and \ref{thm:maineq} are completed. 
\end{remark}

\section*{Acknowledgements}
We thank Luigi Ambrosio for an interesting comment on a preliminary version of this work. S.C. gratefully
thanks Dipartimento di Matematica Guido Castelnuovo, Sapienza Università di Roma, where a large part of this work was carried out, for the warm hospitality. 
This work was partially funded by the Deutsche Forschungsgemeinschaft (DFG,
German Research Foundation), project CRC 1720 -- 539309657.
V.C., A.G., A.M.~acknowledge the support of the MUR - PRIN project 2022J4FYNJ CUP B53D23009320006 ``Variational methods for stationary and evolution problems with singularities and interfaces'', PNRR Italia Domani, funded by the European Union via the program NextGenerationEU. V.C.\ ~acknowledges also the support of the SEED-PNR Project CUP B87G22001200001.  V.C., A.G., A.M. are members of the Gruppo Nazionale per l'Analisi Matematica, la Probabilità e le loro Applicazioni (INdAM-GNAMPA).

\end{document}